\documentclass[11pt]{amsart}
\usepackage[paperheight=11in,
paperwidth=8.5in,left=1.25in,top=1.25in,right=1.25in,bottom=1.25in,]{geometry}

\usepackage{amssymb, amsthm}
\usepackage{amsmath}
\usepackage{bbm}
\usepackage[dvips]{graphicx}
\usepackage{hyperref}
\usepackage[dvipsnames]{xcolor}
\usepackage{multirow}
\usepackage{todonotes}
\usepackage[capitalize,nameinlink,noabbrev,nosort]{cleveref}

\usepackage{tikz}
\usetikzlibrary{arrows}
\usetikzlibrary{decorations.markings}
\tikzset{invmidarrow/.style={postaction=decorate,decoration={markings,mark={between positions 0 and 1 step #1
with {\arrow{latex reversed}}}}}}
\usetikzlibrary{shapes.geometric}

\usepackage{makecell}

\numberwithin{equation}{section}

\newcommand{\ds}{\displaystyle}
\newcommand{\NN}{\mathbb{N}}
\newcommand{\arm}{{\operatorname{arm}}}
\newcommand{\leg}{{\operatorname{leg}}}
\newcommand{\aver}[1]{\langle #1 \rangle}
\renewcommand{\P}{\mathcal{P}}
\DeclareMathOperator{\rev}{rev}
\newcommand{\ASEP}{{F}}
\DeclareMathOperator{\wt}{wt}
\DeclareMathOperator{\inc}{inc}

\newtheorem{thm}{Theorem}[section]
\newtheorem{prop}[thm]{Proposition}
\newtheorem{cor}[thm]{Corollary}
\newtheorem{lem}[thm]{Lemma}
\newtheorem{notation}[thm]{Notation}
\newtheorem{rem}[thm]{Remark}

\theoremstyle{definition}
\newtheorem{defn}[thm]{Definition}
\newtheorem{eg}[thm]{Example}

\title
[The inhomogeneous $t$-PushTASEP and Macdonald polynomials]
{The inhomogeneous $t$-PushTASEP and Macdonald polynomials 
at $q=1$
}
\author{Arvind Ayyer, James Martin, Lauren Williams}
\address{Arvind Ayyer, Department of Mathematics, 
Indian Institute of Science, Bangalore  560012, India.}
\email{arvind@iisc.ac.in}

\address{James Martin, Department of Statistics, 
University of Oxford, UK.}
\email{martin@stats.ox.ac.uk}

\address{Lauren Williams, Department of Mathematics, 
Harvard University, Cambridge MA, 02138.}
\email{williams@math.harvard.edu}


\date{\today}

\begin{document}

\begin{abstract}
We study a {multispecies} $t$-PushTASEP system on a finite
ring of $n$ sites with site-dependent rates $x_1,\dots,x_n$.
Let $\lambda=(\lambda_1,\dots,\lambda_n)$ be a partition
whose parts represent the species of the $n$ particles
on the ring.  
We show that for each composition $\eta$ obtained by permuting
the parts of $\lambda$, the stationary 
probability of being in state $\eta$ is proportional to the ASEP 
polynomial $F_{\eta}(x_1,\dots,x_n; q,t)$ at $q=1$; 
the normalizing constant (or partition function) is
the Macdonald polynomial $P_{\lambda}(x_1,\dots,x_n;q,t)$
at $q=1$. Our approach involves
new relations between the families of ASEP polynomials
and of non-symmetric Macdonald polynomials at $q=1$. 
We also use \emph{multiline diagrams}, showing that
a single jump of the PushTASEP system is closely
related to the operation of moving from one line to the next
in a multiline diagram. 
We derive symmetry properties
for the system under permutation of 
its jump rates, as well as a formula for the 
current of a single-species system.
\end{abstract}

\maketitle

\setcounter{tocdepth}{1}
\tableofcontents

\section{Introduction}

Multispecies versions of the  \textit{asymmetric simple exclusion process} (ASEP) and its relatives have been the subject
of intense study in recent years, from diverse perspectives 
in physics, probability, algebra, and combinatorics. 
The connection between the multispecies ASEP on a ring and Macdonald
polynomials was developed by Cantini, de Gier and Wheeler
\cite{cantini-degier-wheeler-2015}
and Chen, de Gier and Wheeler \cite{chen-degier-wheeler-2020}.
In these works, they define the family of 
\textit{ASEP polynomials}, which are polynomials
in variables $x_1, \dots, x_n$ whose coefficients 
are rational functions in $q$ and $t$. When specialised
to $q=1$ and $x_1=x_2=\dots=x_n$, the ASEP polynomials
describe the stationary distribution of a multispecies 
ASEP on a ring with $n$ sites. 
The ASEP polynomials are in fact 
special cases
of the 
\emph{permuted-basement Macdonald polynomials}
introduced in 
\cite{ferreira-2011},
as shown in 
\cite{corteel-mandelshtam-williams-2022}. 

A construction of the stationary distribution 
of the multispecies ASEP in terms
of \textit{multiline diagrams} was given in \cite{martin-2020},
building on the construction for the TASEP by 
Ferrari and Martin \cite{ferrari-martin-2007}
and the matrix product representation for the ASEP
given by Prolhac, Evans and Mallick \cite{prolhac-evans-mallick-2009}. Corteel, Mandelshtam and Williams
\cite{corteel-mandelshtam-williams-2022} then showed
that a generalisation of the multiline diagrams 
from \cite{martin-2020} could be used to give 
a combinatorial formula for the ASEP polynomials 
with general $x_1,\dots, x_n$, $q$ and $t$. 

The description of the mASEP stationary distribution
in terms of ASEP polynomials with identical $x_i$
invites the question: is there a 
natural multitype particle system, with 
inhomogeneous (i.e.\ site-dependent)
jump rates,
whose stationary probabilities 
are given by the ASEP polynomials with general $x_i$?
An inhomogeneous version of the ASEP itself is not believed to have
nice algebraic properties. 
The main result of this article is that a related process, 
the 
\emph{multispecies $t$-PushTASEP with inhomogeneous rates}, 
does indeed have its stationary distribution given 
by the ASEP polynomials with general $x_i$ -- see 
\cref{thm:multilrep-rej-statprob}.

Our approach involves new relations between the families
of ASEP polynomials and of non-symmetric Macdonald polynomials
at $q=1$,
building on the work of Alexandersson and Sawhney
\cite{alexandersson-sawhney-2019}.
Among other results, we show that certain ratios of 
non-symmetric Macdonald polynomials become symmetric
in the particular case $q=1$. 
We also use the multiline diagram construction
-- a single jump of the PushTASEP system
is closely related to the operation moving from 
one line to the next in a multiline diagram.

Systems related to the multispecies 
$t$-PushTASEP have previously appeared in various contexts.
The multispecies system in the case $t=0$ with homogenous
rates was already studied by Ferrari and Martin
\cite{ferrari-martin-2006},
under the name of \emph{discrete-space Hammersley-Aldous-Diaconis process}
(or long-range exclusion process~\cite{spitzer-1970}). 
A related process in discrete time (dubbed the ``frog model'')
defined on the ring was recently used by Bukh and Cox 
to study problems involving the longest common subsequence between a periodic word and a word with i.i.d.\ uniform entries. 
Moving to the inhomogeneous case,  
a single-type PushTASEP in the case $t=0$ on 
the half-line was considered by Petrov \cite{petrov-2020}, and
the multi-type $t>0$ case on a finite interval 
has been investigated by Borodin and Wheeler \cite[Section 12.5]{borodin-wheeler-2022}
in the context of the coloured stochastic six-vertex model

Most recently, in independent work, Aggarwal, Nicoletti and Petrov 
\cite{aggarwal-nicoletti-petrov-2023} obtain closely related results. They write the stationary distribution of
the multitype inhomogeneous PushTASEP (and other related models including the mASEP and the multi-type TAZRP) in terms of vertex models, which are closely related to multiline diagrams and to matrix product
formulae. Their approach is entirely different to ours, 
making extensive use of Yang-Baxter interchange relations. 

In a companion paper \cite{ayyer-martin-2023}, 
two of the authors focus on the particular case $t=0$.
We employ more direct probabilistic methods involving 
time-reversal and coupling to connect the
stationary distribution to multiline diagrams, and we
describe symmetry properties under permutation
of the rates, which apply to evolutions of
the system out of equilibrium as well as to 
the stationary distribution. 

\subsection{Definition of the $t$-PushTASEP}
\label{sec:dynamics}

In this paper we study the 
inhomogeneous $t$-PushTASEP on a ring with $n$ sites, which 
generalizes the PushTASEP studied 
in~\cite{ayyer-martin-2023}.

A configuration of the system 
is a vector (or \emph{composition}) $(\eta_1, \dots, \eta_n)$ 
whose entries are non-negative 
integers. The entry $\eta_j$ denotes the species
of the particle at site $j$. If two particles have species
$i_1$ and $i_2$ with $i_1>i_2$, we say that the particle 
of species $i_1$ is \emph{stronger} and the particle of 
species $i_2$ is \emph{weaker}.
We often refer to particles
of species $0$ as \emph{holes} or \emph{vacancies}.

The $t$-PushTASEP dynamics will preserve the number of 
particles of each species, 
so we may take the state-space of the system
to be $S_\lambda=S_n(\lambda)$, the set of compositions
which are permutations of some given
\emph{partition} $\lambda=(\lambda_1,\dots, \lambda_n)$
with $\lambda_1\geq\dots\geq\lambda_n\geq 0$. We can
describe such a partition by its vector of types 
$\mathbf{m}=(m_0, m_1, \dots, m_s)$, where 
$m_i=\#\{j:\lambda_j=i\}$ gives 
the number of particles of species $i$,
and where $s$ is the largest species in the system.
Sometimes we denote our partition by 
$\lambda=\langle s^{m_s}, \dots, 1^{m_1}, 0^{m_0} \rangle$.
We have $\sum_{i=0}^s m_i=n$.
We will always require that $m_0 \geq 1$, 
i.e.\ that the system has at least one vacancy.
The partition $\lambda$ is called the \emph{content}
of the system.

The system has positive real parameters $x_1, \dots, x_n$. 
We first define the 
transitions of the PushTASEP, i.e.\ the case $t=0$.
For each site $j$, there is an exponential clock which rings with rate $1/x_j$.
The effect of a bell ringing at site $j$ is as follows. 
If site $j$ contains a vacancy then nothing changes. 
If instead site $j$ contains a particle of species $r_0>0$, 
this particle becomes ``active''. 
It moves clockwise around the ring until 
it finds a site $j_1$ with a particle of smaller species
$r_1<r_0$. The active particle now settles at site $j_1$.
If in fact $r_1=0$ (i.e.\ the site $j_1$ was previously vacant)
then the procedure stops; otherwise the particle
of label $r_1$ becomes active and itself starts to move
clockwise around the ring looking for a site
with a particle of smaller species $r_2<r_1$. Such a procedure
continues until a vacancy is found. {All the transitions 
occur simultaneously and the original site $j$ becomes vacant at the end
of the transition}.

The $t$-PushTASEP is a generalization of the PushTASEP,
with an additional parameter $t$ which for convenience we
take to be in $[0,1)$ (though it is easy to extend to $t\geq 1$). 
Again each site $j$ has a bell ringing at rate $1/x_j$,
and we describe the effect of such a bell. 
If site $j$ is vacant then nothing changes. 
Otherwise, as above, the particle of type $r_0>0$ 
at site $j$ becomes ``active'' and will move to the location
of a weaker particle. However, for $t>0$ the move is not
deterministic. Suppose there are $m$ particles in the system
whose species is less than $r_0$ (including vacancies). 
Recall that $[m]_t = 1 + t + \cdots + t^{m-1}=\frac{1-t^m}{1-t}$ denotes 
the $t$-analogue of the integer $m$.
Then the particle at site $j$ 
will travel clockwise around the ring, and
with probability $t^{k-1}/[m]_t$, it will
move to the location of the $k$'th of these lower-species 
particles. 
If this location is not vacant, then 
the particle there becomes active, and
chooses a weaker particle to displace in the same way. 
The procedure continues until a vacancy is chosen. 
All these transitions occur simultaneously. Again
the site $j$ itself always becomes vacant {at the end of the transition}.

\begin{figure}[ht]
\begin{center}
\begin{minipage}{0.45\textwidth}
\[
\begin{array}{|c|c|}
\hline
\text{Configuration} & \text{Rate} \\
\hline &\\[-0.2cm]
(2, 4, \underline{0}, \underline 3, 2, 4, 1, 3) & \ds \frac{1}{x_3} \frac{1}{[4]_t}\\[0.5cm]
(2, 4, \underline 0, \underline 1, \underline 3, 4,  \underline 2, 3) & \ds \frac{1}{x_3} \frac{t}{[4]_t} \frac{1}{[2]_t}\\[0.5cm]
(2, 4, \underline 0, \underline 2, \underline 3, 4, 1, 3) & \ds \frac{1}{x_3}\frac{t}{[4]_t} \frac{t}{[2]_t}\\[0.5cm]
(2, 4, \underline 0, \underline 1, 2, 4, \underline 3, 3) & \ds \frac{1}{x_3}\frac{t^2}{[4]_t}\\[0.5cm]
(\underline 3, 4, \underline 0, \underline 2, 2, 4, 1, 3) & \ds \frac{1}{x_3}\frac{t^3}{[4]_t} \frac{1}{[2]_t}\\[0.5cm]
(\underline 3, 4, \underline 0, \underline 1, 2, 4, \underline 2, 3) & \ds \frac{1}{x_3}\frac{t^3}{[4]_t} \frac{t}{[2]_t}\\[0.5cm]
\hline
\end{array}
\]
\end{minipage}
$\vcenter{\hbox{
\includegraphics[height=8cm]{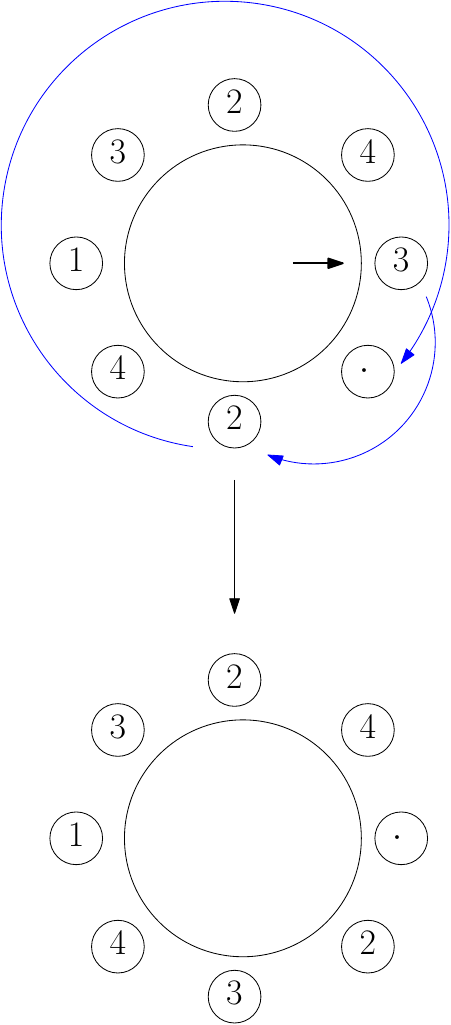}
}}$
\caption{
\label{fig:transition}
Let $\eta = (2, 4, 3, 0,2, 4, 1, \allowbreak 3)$ with $n=8$ and $s=4$. If the bell rings at site $3$,  some particles will move -- 
the table  shows the possible destination
configurations, along with the rate of the jump to each one.
 In each case the particles which moved are underlined. 
The transition corresponding to the $4$th line of the
table is illustrated on the right. Site $1$ is shown at the top of the ring, and site $3$ where the bell rings is on the extreme right.
}
\end{center}
\end{figure}

We may interpret the procedure above by saying that
the active particle moves clockwise around the ring
looking for a weaker particle to displace, but rejects
each option with probability $t$. 
See \cref{sec:project} for an equivalent definition
along these lines. 

See \cref{fig:transition} for examples
of transitions from a given configuration on a ring of size $8$. 
The  state diagram for the system defined by 
$\lambda=(2,1,0)$, i.e.\
$\mathbf{m} =  (1,1,1)$,
is given in \cref{fig:eg123t}.

\begin{figure}[h!]
\begin{center}
\includegraphics[scale=0.65]{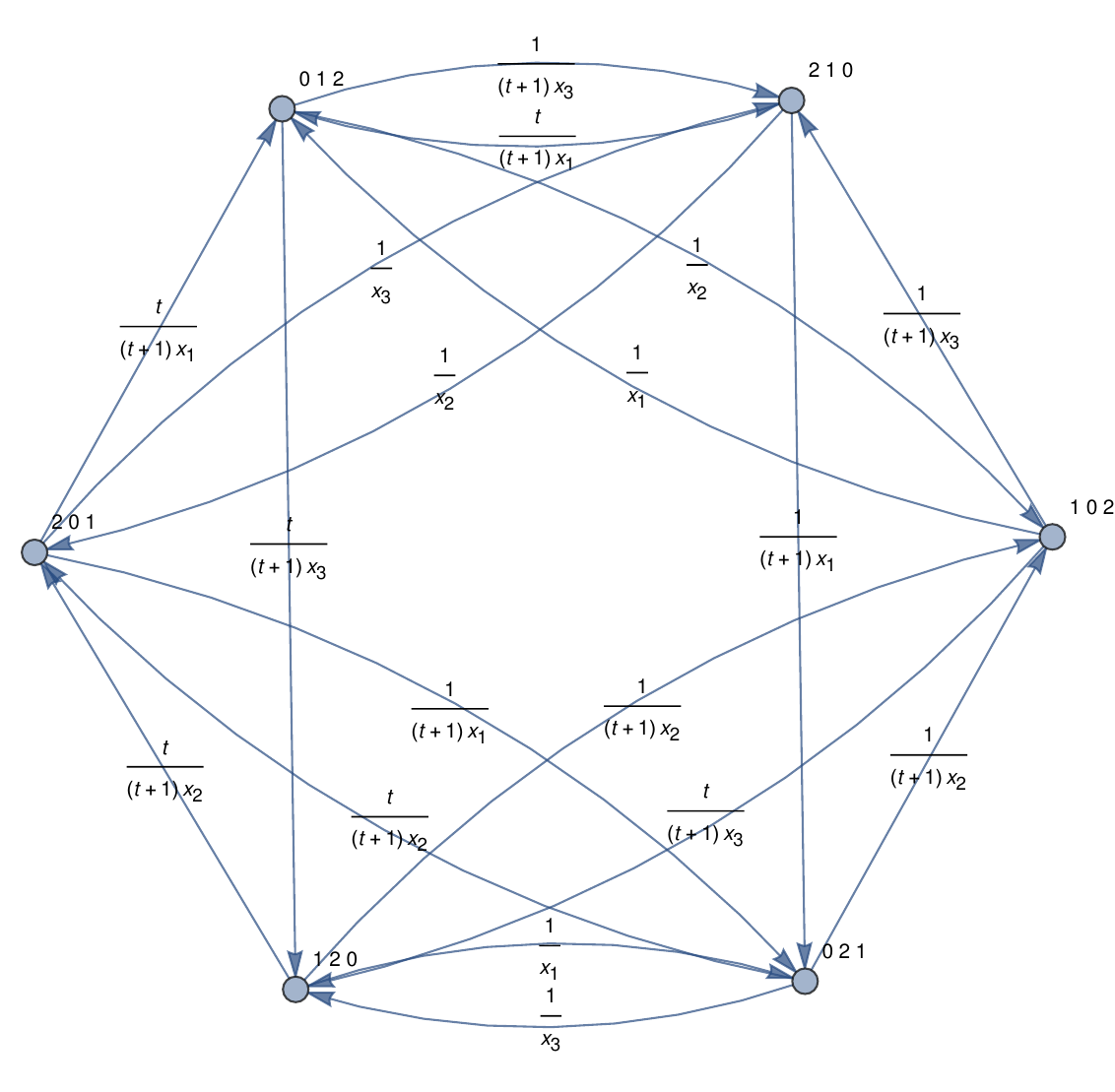}
	\caption{The transition graph of the multispecies $t$-PushTASEP for $\mathbf{m} = (1,1,1)$.}
\label{fig:eg123t}
\end{center}
\end{figure}

\subsection{Main results}

Our main result is a description of the 
stationary distribution of a $t$-PushTASEP system 
in terms of the \emph{ASEP polynomials}, cf. 
	\cref{def:ASEPpolynomials}.

\begin{thm}
\label{thm:multilrep-rej-statprob}
In the multispecies $t$-PushTASEP with content $\lambda=(\lambda_1,\dots,\lambda_n)$
	and parameters $\mathbf{x} = (x_1,\dots,x_n)$, 
the stationary probability of
a configuration $\eta\in S_{\lambda}$ is given by
\[
	\pi_{\lambda}(\eta) = \frac{\ASEP_\eta(\mathbf{x};1, t)}{P_\lambda(\mathbf{x}; 1, t)},
\]
where $\ASEP_\eta(\mathbf{x}; q, t)$ 
	is the ASEP polynomial 
from 	\cref{def:ASEPpolynomials} associated to $\eta$,
	and 
	$P_{\lambda}(\mathbf{x}; q, t)$ 
	is the Macdonald polynomial associated to $\lambda$.
\end{thm}

\begin{eg}
\label{eg:210 probs}
The steady state probabilities
for the example from \cref{fig:eg123t} are proportional to 
	the ASEP polynomials $F_{\eta}$ at $q=1$, which are given in the following table.
	(\cref{eg:F} shows the same ASEP polynomials but for general $q$.)
The sum of these polynomials is the Macdonald polynomial
\[
P_{(2,1,0)}(x_1, x_2, x_2; 1, t) =  \left(x_1+x_2+x_3\right) \left(x_1 x_2 + x_1 x_3 + x_2 x_3\right),
\]
which can be seen as a partition function for the system;
note that it is independent of $t$ (see \eqref{Pla-q=1}).
\[
\begin{array}{|c|c|}
\hline
\eta & \ASEP_\eta(x_1, x_2, x_2; 1, t) \\
\hline
(2,1,0) & \ds x_1 x_2 \left(x_1 + \frac{x_3}{1+t} \right)\\ 
(2,0,1) & \ds x_1 x_3 \left(x_1 + \frac{x_2 t}{1+t} \right)\\ 
(1,2,0) & \ds x_1 x_2 \left(x_2 + \frac{x_3 t}{1+t} \right)\\ 
(1,0,2) & \ds x_1 x_3 \left(\frac{x_2}{1+t} + x_3 \right)\\ 
(0,2,1) & \ds x_2 x_3 \left(\frac{x_1}{1+t} + x_2\right)\\ 
(0,1,2) & \ds x_2 x_3 \left(\frac{x_1 t}{1+t} + x_3 \right) \\
\hline
\end{array}
\]
\end{eg}

A combinatorial formula for the ASEP polynomials in terms of multiline diagrams 
was given in 
 \cite{corteel-mandelshtam-williams-2022}.
Combining this with 
\cref{thm:multilrep-rej-statprob}, we obtain
 the following  corollary. 
\begin{cor}
Consider a multiline diagram
as defined in \cite{corteel-mandelshtam-williams-2022}
of type $\lambda$,
with parameters $x_1,\dots, x_n, t$ and $q=1$. 
The distribution of the bottom line of 
the diagram is the same as the stationary distribution
of the $t$-PushTASEP with content $\lambda$. 
\end{cor}
We will define multiline diagrams in \cref{sec:generate}
for the special case where $\lambda$ has all parts distinct. 
See \cite{corteel-mandelshtam-williams-2022} for
the general definition.

From \cref{thm:multilrep-rej-statprob},
we can also derive a symmetry property
for the $t$-PushTASEP under permutation of the jump-rate
parameters $x_i$. If $O$ is an observable (an event or a
random variable),  write 
$\aver{ O }$ for its probability or expectation 
in the stationary distribution.
\begin{thm}
\label{thm:obs}
Fix $k < n$ and let $O$ be any observable 
in the stationary distribution of  multispecies $t$-PushTASEP
which depends only on the 
configuration in sites $1,2,\dots, k$. 
Then $\aver{ O }$ is
symmetric in the parameters $x_{k+1}, \dots, x_n$.
\end{thm}

Note that in \cite{ayyer-martin-2023}, 
a symmetry result which is stronger 
than \cref{thm:obs}
is proved in the case $t=0$.
That result extends also
to observables depending on the \textit{path}
of the process (not just its state at a single time), 
and also to processes out of equilibrium (if started from
suitable initial states). Whether this stronger symmetry property
also holds for $t>0$ is an interesting open question 
(see e.g.\ \cite{ayyer-mandelshtam-martin-2022} for related
discussions in the case of the totally asymmetric zero-range process).

We now consider other important quantities for the $t$-PushTASEP in its stationary distribution. Two natural such quantities are the density of a particular species, which is the probability of seeing a particle of that species at some site, and the current, which is the number of particles of a given species crossing an edge per unit time. 

We will show in 
\cref{cor:dens}
that the formulas for the density are independent of $t$, and hence the same as that for $t = 0$ given in \cite{ayyer-martin-2023}. The case of the current is much more interesting. It 
turns out that even when we have only a single species, the formula is nontrivial. We will prove the following result.

\begin{thm}
\label{thm:curr-singlePushTASEP}
The current between two adjacent sites 
	(say sites $n$ and $1$)  in the stationary distribution of the single species $t$-PushTASEP on $S_{\langle 1^{m_1}, 0^{m_0} \rangle}$ is given by
	\begin{align*}
J_{m_0, m_1}  
		&= \frac{1 + 2t + 3t^2 + \cdots + m_0 t^{m_0-1}}{1 + t + \cdots + t^{m_0-1}} \cdot
\frac{e_{m_1-1}(x_1,\dots, x_{n}) }{e_{m_1}(x_1,\dots,x_{n})}\\
		&= \frac{ \frac{d}{dt} (1 + t + t^2 + \cdots +  t^{m_0})}{1 + t + \cdots + t^{m_0-1}} \cdot
\frac{e_{m_1-1}(x_1,\dots, x_{n}) }{e_{m_1}(x_1,\dots,x_{n})},
	\end{align*}
\end{thm}
The $t = 0$ case of this result was proved in \cite{ayyer-martin-2023} using a coloring argument. We will prove the result in \cref{sec:current}.
Generalizing the formula for the current to the multispecies $t$-PushTASEP seems considerably harder, and in \cref{sec:current} we explain why the coloring approach does not work in that case.

The structure of this paper is as follows.
In \cref{sec:push} we discuss some basic properties of the $t$-PushTASEP,
including an important recoloring property. 
In \cref{Macdonald} we 
provide background on nonsymmetric Macdonald polynomials, ASEP 
polynomials, (symmetric) Macdonald polynomials, and permuted basement Macdonald polynomials.
In \cref{sec:symm-fn}
we prove some properties
 of nonsymmetric Macdonald polynomials
and ASEP polynomials at the specialization $q=1$, in particular 
\cref{prop:f-projection}, which will 
be a main ingredient in our proof of 
\cref{thm:multilrep-rej-statprob}.
In \cref{sec:asep}
we define multiline diagrams 
and explain their relation to the $t$-PushTASEP.
In \cref{sec:proof}
we prove our main result, \cref{thm:multilrep-rej-statprob}.
We end the paper with 
the formulas for the density and the current in stationarity in 
\cref{sec:current}.

\subsection*{Acknowledgements}
{We thank Per Alexandersson, Gidi Amir, Luigi Cantini, Pablo Ferrari, Jan de Gier, Svante Linusson, Leo Petrov, and Michael Wheeler for helpful discussions.}  We especially thank Omer Angel for valuable discussions while 
this project was in its formative stages.
AA acknowledges support from the DST FIST program - 2021 [TPN - 700661].
and by SERB Core grant CRG/2021/001592.
LW is supported by the National Science Foundation under Award No.
DMS-2152991. Any opinions, findings, and conclusions or recommendations expressed in this material are
those of the author(s) and do not necessarily reflect the views of the National Science
Foundation.

\section{Basic properties of the $t$-PushTASEP}\label{sec:push}

In this section we discuss some basic properties
of the $t$-PushTASEP system which we will 
rely on in the later analysis.

\subsection{Projections and couplings}
\label{sec:projections}
The multispecies dynamics defined above have an 
important ``recolouring'' property. If we relabel 
the particles while (weakly) preserving the order
of the labels, the resulting system still follows
$t$-PushTASEP dynamics. This allows us to project 
from a ``finer'' multispecies system to a
``coarser'' one, by merging groups of two or 
more adjacent species into one. 

As an extreme case, we can consider all 
particles of species $i,\dots, s$  as ``particles''
(with the new label $1$) and all 
particles of  species $1,\dots,i-1$ as vacancies
(with the new label $0$), to obtain a single-species
process with a total of $a_i:=m_i+\dots+m_s$
particles and $n-a_i$ vacancies. 
Considering such projections for all $i=1,2,\dots, s$, 
we can identify the multispecies process as a coupling of $s$ 
single-species processes. This is a version of the \textit{basic coupling} \cite[Chapter VIII, Section 2]{liggett-1985} (under which 
the bells ring at the same sites at the same types in 
all the coupled single-species systems).

To state the above projection (or lumping) properties precisely,
we make the following definition.

\begin{defn}\label{def:orderpreserving}
We say that a function $\phi$ from $\NN$ to $\NN$ 
is \emph{weakly order-preserving} if $\phi(i)\leq\phi(j)$ whenever
$i\leq j$. For such a function $\phi$ and a composition 
$\rho=(\rho_1, \dots, \rho_n)$, 
define $\phi(\rho)$ componentwise by $\phi(\rho)=(\phi(\rho_1),
\dots, \phi(\rho_n))$. 
\end{defn}
For example, the function $\phi$ which sends elements of 
$\{0,1,2,3,4\}$ to $2$, elements of $\{5,6\}$ to $4$,
and is the identity otherwise, is a 
weakly order-preserving function.
Note that if $\rho$ is a partition 
then so is $\phi(\rho)$.

\begin{prop}
\label{prop:colouring}
Let $\phi:\NN\mapsto\NN$ be a weakly order-preserving function 
with $\phi(0)=0$.
Consider a 
multispecies $t$-PushTASEP process with content given by the
partition $\lambda$. Via the map $\phi$, this process
projects to a multispecies $t$-PushTASEP with content $\mu$,
where $\mu=\phi(\lambda)$. 
\end{prop}

\begin{proof}
Recall the description of the $t$-PushTASEP dynamics from
\cref{sec:dynamics}. Consider an active particle
of species $r$, in a system with $m$ particles weaker than $r$ (including vacancies). For $1\leq k\leq m$, the particle
moves to the $k$th out of the $m$ locations containing such a weaker particle (considered in order clockwise from its current location)
	with probability $\frac{t^{k-1}}{1+t+\dots+t^{m-1}}$ -- it displaces
the particle currently occupying that site, which itself becomes 
active.

We may alternatively describe the procedure as follows. The particle
moves clockwise around the ring, and each time
it passes a site with a weaker particle, it settles at that site 
with probability $1-t$, and continues moving with probability $t$. 
If it passes the $m$th such site, then it continues cyclically
around the ring, with the $(m+1)$st option it considers being
the same as the first, and so on. 
Hence for $1 \leq k\leq m$, it chooses
the $k$th available option with probability 
\begin{align*}
(1-t)(t^{k-1}+t^{k-1+m} + t^{k-1+2m} + \dots) &= 
(1-t) t^{k-1}(1+t^m+t^{2m}+ \dots)\\
&= \frac{t^{k-1}}{1+t+ \dots + t^{m-1}},
\end{align*}
as in our original description of the procedure.

We make use of one further freedom -- when the active particle
passes a site containing a particle with the same label, 
it makes no difference whether we allow the active particle
to displace its ``twin'' or not. 

This description allows us to maintain a coupling between
a system of particles $(\eta(u),u\geq 0)$ 
with content $\lambda$ and a system of particles
$(\zeta(u), u\geq 0)$  with 
content $\mu$, such that at all times $u\geq 0$,
$\zeta(u)=\phi(\eta(u))$, i.e.\ 
$\zeta_j(u)=\phi(\eta_j(u))$ for all $1\leq j\leq n$. 
The bells ring at the same time in 
both systems. When a bell rings at some site in the 
$\lambda$-system currently in configuration $\eta$, we observe some collection of transitions of particles according to the description above, with each active
particle moving clockwise and settling on an available location with probability $1-t$. Under the coupling, if a particle from some site $j$ moves to some new site $j'$ in the $\lambda$-system, exactly the same 
will occur in the $\mu$-system. Since $\eta(j)>\eta(j')$
and $\phi$ is weakly order-preserving, we have that
$\phi(\eta(j))\geq\phi(\eta(j'))$, and the same transition is
possible in the $\mu$-system as required. 

Note that this can also be interpreted as a commutation property. 
Let $\eta\in S_\lambda$ be a configuration of the $\lambda$-system,
and let $j$ be any site.
Consider the follow two operations to obtain $\zeta\in S_\mu$:
(a) generate a configuration $\eta'$ resulting from the
ring of a bell at site $j$, and then recolour $\eta'$ by $\phi$
to obtain a state $\zeta=\phi(\eta')$;
(b) recolour $\eta$ to give $\zeta'=\phi(\eta)$,
and then generate a configuration $\zeta$ from $\zeta'$
resulting from the ring of a bell at site $j$.
The coupling above shows that (a) and (b) lead to the
same distribution of $\zeta\in S_\mu$. This commutation
gives the required projection property.
\end{proof}

We can then immediately deduce a corresponding
recolouring property for the stationary distributions:
\begin{prop}
\label{prop:colouring-stationary}
Let $\phi:\NN\mapsto\NN$ be a weakly order-preserving function
with $\phi(0)=0$, 
and suppose we have partitions $\mu$ and $\lambda$ with $\phi(\lambda)=\mu$. 
	Let $\pi_{(\lambda)} = (\pi_{(\lambda)}(\eta)$, $\eta\in S_\lambda)$ 
		and $\pi_{(\mu)} = (\pi_{(\mu)}(\eta)$, $\eta\in S_\mu)$ 
denote the stationary distributions of the $t$-PushTASEP
		with content $\lambda$ and $\mu$.
Then for all $\eta\in S_\mu$, 
\begin{equation}\label{projection-stationary}
	\pi_{(\mu)}(\eta)
=
\sum_{\zeta\in S_\lambda: \phi(\zeta)=\eta}
	\pi_{(\lambda)}(\zeta).
\end{equation}
		\end{prop}

\subsection{Transition rates}
For completeness and later use, we give here a direct description of the 
transition rates for the $t$-PushTASEP dynamics.
Let $\lambda$ be a partition, and as above let
$m_i$ be the number of entries $i$ in $\lambda$.

Let $\eta, \eta'\in S_\lambda$. 
It follows from the description in \cref{sec:dynamics}
that a bell at $j$ can cause a transition 
from the configuration $\eta$ to the configuration $\eta'$
precisely if the following conditions are satisfied:
\begin{itemize}
\item
$j$ is the unique site which is vacant in $\eta'$ and not vacant in $\eta$. For every other site 
$i$, $\eta'(i)\geq \eta(i)$. 
\item
For each type $h>0$ with $m_h>0$, either:
\begin{enumerate}
\item
the sites occupied by species $h$ are the same in $\eta$ and $\eta'$; or,
\item
There exists exactly one site $j(h)$ such that $\eta_{j(h)}=h$ and 
$\eta'_{j(h)}\ne h$. It follows that there 
also exists exactly one site $j'(h)$ such that $\eta'_{j'(h)}=h$ and 
$\eta_{j'(h)}\ne h$.
\end{enumerate}
\end{itemize}
Define $w_{\eta,\eta'}(h)$ for each $h$ as follows. If case (1) holds
then $w_{\eta,\eta'}(h)=1$. If case (2) holds then let $K_h$ be the number of 
entries of $\lambda$
smaller than $h$ (including zeros). 
Let $\ell_h$ be the number of sites in the cyclic interval $(j(h), j'(h))$, excluding endpoints, with value smaller than $h$ in $\eta'$. 
Let 
\[
\label{whdef}
w_{\eta,\eta'}(h)=\frac{t^{\ell_h}}{1+t+\dots+t^{K_h-1}}.
\]
Suppose the system is in state $\eta$.
When a bell rings at $j$, 
a jump occurs to $\eta'$ with
probability 
\begin{equation}\label{jump probability}
\prod_{h>0 :m_h>0} w_{\eta,\eta'}(h).
\end{equation}
(The species $h$ for which case (2) holds above
are precisely those for which some particle
of species $h$ becomes ``active''
during the transition, in the sense
of \cref{sec:dynamics}.)
The transition rate from $\eta$ to $\eta'$ is
therefore
\[
\frac{1}{x_j} \prod_{h> 0:m_h>0} w_{\eta,\eta'}(h).
\]

\subsection{Single-species stationary distributions} \label{sec:single}

In the case that we have only one species of particle,
it turns out that the stationary distribution is independent of $t$.  It thus
 matches the distribution when $t = 0$ given in \cite{ayyer-martin-2023}.

\begin{prop}
\label{prop:ss-singlePushTASEP}
Let $\lambda = \langle 1^{m_1}, 0^{m_0} \rangle$ where $m_1+m_0=n$, and 
define 
\begin{equation}\label{eq:single}
\pi(\eta) := \frac{1}{e_{m_1}(x_1,\dots,x_{n})} \prod_{\substack{i=1 \\ \eta_i = 1} }^{n} x_i.
\end{equation}
Then the
	stationary probability  of $\eta \in S_{\lambda}$
	for the $t$-PushTASEP is $\pi(\eta)$.
\end{prop}

\begin{proof}
Since the $t$-PushTASEP is irreducible, it suffices to verify the 
global balance equations. The total weight of outgoing transitions
from the configuration $\eta$ which involve a particle at site $j$ is $\pi(\eta)/x_j$, since the particle at site $j$ makes a transition to some vacancy with rate $1/x_j$.

As for the incoming rate to $\eta$, note that a configuration $\tau$ makes a transition to $\eta$ if there exist a pair of positions $j \neq k$ such that $\eta_j=1$, $\eta_k=0$,
$\tau_j=0$, and $\tau_k=1$, and $\eta$ and $\tau$ agree outside of positions
$j$ and $k$.
Then $\pi(\eta)/\pi(\tau) = x_j/x_k$.
If there are $a$ vacancies strictly between positions $k$ and $j$ (traveling in the clockwise direction starting at $k$), then 
the weight of the transition from $\tau$ to $\eta$ 
is $\pi(\tau) \cdot \frac{1}{x_k} \cdot \frac{t^a}{[m_0]_t} = 
\pi(\eta) \cdot \frac{1}{x_j} \cdot \frac{t^a}{[m_0]_t}.$
	But now if we fix $j$ and sum over all possible $k$ (and corresponding
	$\tau$), the weight of all these transitions to $\eta$
	will be $\pi(\eta) \cdot \frac{1}{x_j}$.
This is exactly the same as the total weight of outgoing transitions from $\eta$ involving the particle at site $j$, as argued above. Summing over all possible locations of particles completes the proof.
\end{proof}

\cref{prop:ss-singlePushTASEP} will be useful when we discuss the generation of
multiline diagrams in 
\cref{sec:generate}.  We will also use this result to analyze the density of particles
in \cref{sec:current}.

\section{Background on Macdonald and ASEP polynomials}\label{Macdonald}

In this section we define nonsymmetric Macdonald polynomials, ASEP 
polynomials, and (symmetric) Macdonald polynomials.  We also mention
some relations with permuted basement Macdonald polynomials.  All of 
the above polynomials are elements of the polynomial ring 
$\mathbb{Q}(q,t)[x_1, \dots, \allowbreak x_n]$
 in variables $x_1, \dots, x_n$,
 with coefficients in $\mathbb{Q}(q,t)$.

\subsection{Nonsymmetric Macdonald polynomials}
\label{sec:nonsym-macd}

Nonsymmetric Macdonald polynomial can be defined as eigenfunctions
of the \emph{$q$-Dunkl} or \emph{Cheredik operators}.
We will mostly follow the notation of \cite{corteel-mandelshtam-williams-2022}.

For $f := f(x_1, \dots, x_n; q, t) \in \mathbb{R}$, we define the operators $s_i$ and $L_i$, $1 \leq i \leq n-1$ as
\[
s_i (f) = f(x_1, \dots, x_{i-1}, x_{i+1}, x_i, x_{i+2}, \dots, x_n)
\]
and
\[
L_i (f) = \frac{t x_i - x_{i+1}}{x_i - x_{i+1}} \left( f - s_i (f) \right).
\]
Using these, we define operators $T_i$ and $T_i^{-1}$, $1 \leq i \leq n-1$, as
\begin{equation}
\label{def Ti}
T_i (f) = t f - L_i(f), \qquad T_i^{-1} (f) = t^{-1} f - t^{-1} L_i (f).
\end{equation}
These operators satisfy 
the \emph{Hecke algebra} relations,
\begin{equation}
\label{hecke-gens}
\begin{split}
(T_i - t) (T_i + 1) = 0, & \quad T_i T_{i+1} T_i  = T_{i+1} T_i T_{i+1}, \\
T_i T_j = T_j T_i & \quad \text{if $|i - j| \geq 1$.}
\end{split}
\end{equation}
We also define the shift operator $\omega$ as
\begin{equation}
\label{omega-def}
(\omega f) (x_1, \dots, x_n) = 
f(q x_n, x_1, \dots, x_{n-1}).
\end{equation}
Using the operators $T_i, T_i^{-1}$ and $\omega$
we define the \emph{Cherednik-Dunkl operators} $Y_i$, $1 \leq i \leq n$ by
\[
	Y_i = t^{-(i-1)} T_i^{-1} \cdots T_{n-1}^{-1} \omega T_1 \cdots T_{i-1}.
\]
One can show that these $Y_i$'s mutually commute 
and can therefore be simultaneously diagonalized.

\begin{rem}\label{rem:operator}
Let $h,f\in \mathbb{Q}(q,t)[x_1, \dots, \allowbreak x_n]$
such that $h$ is symmetric in $x_1,\dots,x_n$.
It follows from the definition of the operator $T_i$ 
that $T_i(hf)=hT_i(f)$
for all $i$.
\end{rem}

A (weak) \emph{composition} $\eta = (\eta_1, \dots, \eta_n)$ is a tuple of nonnegative integers. 
Let 
\begin{equation}\label{eq:eigenvalue}
	y_i(\eta; q, t) = q^{\eta_i} \, t^{-| \{j>i \ \vert \ \eta_j \geq \eta_i\}|-|\{j<i \ \vert \ \eta_j>\eta_i\}|}.
\end{equation}

\begin{defn}\label{def:partialorder}
Let $\P_n$ denote the set of partitions  of $n$.
The \emph{dominance order} on partitions, denoted $\leq$,
is a partial order on the partitions in $\P_n$ with 
fixed size, defined as follows. 
We say that $\lambda \leq \mu$ if $\lambda_1 + \cdots + \lambda_i \leq \mu_1 + \cdots + \mu_i$ for $1 \leq i \leq n$.

To extend the dominance order to an order on compositions $\eta$,
let $\eta^+$ be the partition obtained by ordering the parts of $\eta$ in weakly decreasing order.
We write
$\eta \leq \nu$ for compositions $\eta$ and $\nu$ if either
$\eta^+ \leq \nu^+$, or $\eta^+ = \nu^+$ and
$\eta_1 + \cdots + \eta_i \leq \nu_1 + \cdots + \nu_i$ for $1 \leq i \leq n$. 
{Note that this partial order is not the natural generalization of the dominance order to compositions.}
\end{defn}

\begin{defn}\label{def:E}\cite{cherednik-1995b, marshall-1999}
	The \emph{nonsymmetric Macdonald polynomial} $E_{\eta} = E_\eta(\mathbf{x}; q, t)$ 
	associated to a composition $\eta=(\eta_1,\dots,\eta_n)$  is the polynomial uniquely defined by the conditions
\begin{enumerate}
\item $E_\eta(\mathbf{x}; q, t)$ 
has the monomial expansion 
\[
	E_\eta(\mathbf{x};q,t) = \sum_{\zeta \leq \eta } v_{\eta, \zeta}(q, t) \mathbf{x}^\zeta,
\]
		where $\mathbf{x}^\zeta$ is shorthand for $x_1^{\zeta_1} \cdots x_n^{\zeta_n}$, and $v_{\zeta, \zeta} = 1$.

\item $Y_i E_\eta = y_i (\eta; q, t) E_\eta$ for all $1 \leq i \leq n$ and all compositions $\eta$.

\end{enumerate}
\end{defn}
The existence of these polynomials is highly nontrivial.

\begin{rem}\label{rem:eigenvalues}
It follows from \eqref{eq:eigenvalue} that 
two distinct nonsymmetric Macdonald polynomials 
$E_{\eta}$ and $E_{\tau}$ have distinct 
tuples of eigenvalues
$(y_1,\dots,y_n)$ as functions of $q$ and $t$.
However, they may not be distinct as functions of $t$
when we set $q=1$; then for example, $E_{210}$, $E_{211}$, and $E_{100}$ all have the same eigenvalues. 

However, if $\eta$ and $\tau$ are permutations of the same 
partition then it still does hold that the tuples
of their eigenvalues remain distinct under the specialisation
$q=1$. 
\end{rem}

\begin{eg}
\label{eg:nsym-macd}
Consider all permutations of the partition $(2, 1, 0)$ as compositions. The nonsymmetric Macdonald polynomials for each of them,
along with their eigenvalues, are as follows:

\medskip
\begin{tabular}{|c|c|c|c|c|}
\hline
$\eta$ & $E_\eta(x_1, x_2, x_3)$ & $\scriptstyle y_1(\eta)$ & $\scriptstyle y_2(\eta)$ & $\scriptstyle y_3(\eta)$ \\
\hline
$(0, 1, 2)$ &
\makecell{
$\frac{(t-1) \left(q^2 t^3-2 q t^2+t^2-t+1\right)}{(q t-1)^3 (q t+1)} x_{1}^{2} x_{2} 
+ \frac{(t-1)^2 }{(q t-1)^2} (x_{1} x_{2}^{2} + x_{1}^{2} x_{3})
$\\[5pt]
$+ x_{2} x_{3}^{2}
+ \frac{(t-1) \left(q^3 t^3+2 q^2 t^3-3 q^2 t^2-2 q t+q-t+2\right)}{(q t-1)^3 (q t+1)} x_{1} x_{2} x_{3}$
\\[5pt]
$+ \frac{1-t}{1-q t} (x_{2}^{2} x_{3} + x_{1} x_{3}^{2})$
}
& 
$t^{-2}$ &
$qt^{-1}$ &
$q^2$ 
\\
\hline
$(0, 2, 1)$ &
\makecell{
$\frac{(t-1)^2}{(q t-1)^2 (q t+1)} x_{1}^{2} x_{2} 
+ \frac{1-t}{1-q t} x_{1} x_{2}^{2} 
+ \frac{1-t}{1-q^{2} t^{2}} x_{1}^{2} x_{3}$ 
\\[5pt]
$+ \frac{(t-1) \left(q^2 t^2+q t-q-1\right)}{(q t-1)^2 (q t+1)} x_{1} x_{2} x_{3} 
+ x_{2}^{2} x_{3}$}
& $t^{-2}$ & $q^2$ & $qt^{-1}$
\\
\hline
$(1, 0, 2)$ &
\makecell{
$\frac{(t-1)^2}{(q t-1)^2 (q t+1)}  x_{1}^{2} x_{2} 
+ \frac{1-t}{1-q^{2} t^{2}} x_{1} x_{2}^{2} 
+ \frac{1-t}{1-q t } x_{1}^{2} x_{3}$ \\[5pt]
$+ \frac{(t-1) \left(q^2 t^2+q t-q-1\right)}{(q t-1)^2 (q t+1)} x_{1} x_{2} x_{3} 
+ x_{1} x_{3}^{2}$
}
&
$qt^{-1}$ &
$t^{-2}$ &
$q^2$ 
\\
\hline
$(1, 2, 0)$ & 
$\frac{1-t}{1-q t} x_{1}^{2} x_{2} 
+ x_{1} x_{2}^{2} 
+ \frac{q (1- t)}{1 -q t} x_{1} x_{2} x_{3}$ 
&
$qt^{-1}$ &
$q^2$ &
$t^{-2}$ 
\\
\hline
$(2, 0, 1)$ & $\frac{1-t}{1-q t} x_{1}^{2} x_{2} 
+ x_{1}^{2} x_{3} 
+ \frac{q (1-t)}{1-q t } x_{1} x_{2} x_{3}$
&
$q^2$ &
$t^{-2}$ &
$qt^{-1}$ 
\\
\hline
$(2, 1, 0)$ & $x_{1}^{2} x_{2} 
+ \frac{q(1- t)}{1-q t^{2}} x_{1} x_{2} x_{3} $
&
$q^2$ &
$qt^{-1}$ &
$t^{-2}$ 
\\
\hline
\end{tabular}

\medskip
\end{eg}

\subsection{ASEP polynomials}

Cantini, de Gier and Wheeler~\cite{cantini-degier-wheeler-2015} 
related Macdonald polynomials $P_\lambda$ at $x_i=1$ and $q=1$ to the multispecies 
ASEP on a ring, via 
the notion of a \emph{qKZ family}, 
which we now explain.

The following notion of \emph{qKZ family} was introduced in
\cite{KasataniTakeyama}, also explaining the relationship of such
polynomials to nonsymmetric Macdonald polynomials.
We use the conventions of \cite[Section 1.3]{cantini-degier-wheeler-2015}.

\begin{defn}\label{qKZ}
Fix a partition $\lambda = (\lambda_1,\dots,\lambda_n)$.
        We say that a family
        $\{f_{\eta} \}_{\eta
        \in S_{\lambda}}$
        of homogeneous
        degree $|\lambda|$
        polynomials in $n$ variables
        $\mathbf{x} = (x_1,\dots,x_n)$, with coefficients which are rational functions of
        $q$ and $t$, is a \emph{qKZ family} if they satisfy
        \begin{align}
T_i f_{\eta}(\mathbf{x}; q,t) &= f_{s_i \eta} (\mathbf{x};q,t),
\text{ when }\eta_i > \eta_{i+1}, \label{firstproperty}\\
        T_i f_{\eta}(\mathbf{x}; q,t) &= t f_{\eta}(\mathbf{x}; q,t), \text{ when }\eta_i = \eta_{i+1}, \label{secondproperty}\\
q^{\eta_n} f_{\eta}(\mathbf{x}; q,t)
&=
                f_{\eta_n,\eta_1,\dots,\eta_{n-1}}(q x_n,x_1,\dots,x_{n-1}; q,t).  \label{thirdproperty}
\end{align}
We say that a family of polynomials is a \emph{KZ family}
if they satisfy the above relations at $q=1$.
\end{defn}
\begin{rem}
Note that \eqref{thirdproperty} can be rephrased as
\[
q^{\eta_n} f_{\eta}(\mathbf{x}; q,t)
=
                (\omega f_{\eta_n,\eta_1,\dots,\eta_{n-1}})(\mathbf{x}; q,t).
                \]
\end{rem}
\begin{rem}
Using the fact that $T_i^2 = (t-1)T_i+t$, together with \eqref{firstproperty}, we 
	see that any qKZ family also satisfies
	\begin{equation}\label{fourthproperty}
		T_i f_{\eta} = (t-1)f_{\eta} + t f_{s_i \eta} \text{ when }
		\eta_i < \eta_{i+1}.
	\end{equation}
\end{rem}

The following polynomials were first 
introduced in \cite{cantini-degier-wheeler-2015}.  They were subsequently
shown to be generating functions for multiline queues in 
\cite{corteel-mandelshtam-williams-2022},
see \cref{thm:ASEP poly combi}.
They were called \emph{ASEP polynomials} by Chen, de Gier and Wheeler~\cite{chen-degier-wheeler-2020}. 

\begin{defn}[ASEP polynomials]\label{def:ASEPpolynomials}
{Given a partition $\lambda$, the}
\emph{ASEP polynomials}
$$\{\ASEP_\eta := \ASEP_\eta(\mathbf{x}; q, t) \ \vert \ \eta\in S_{\lambda}\}$$ are the unique family of polynomials 
	which are a qKZ family and such that $\ASEP_{\lambda}(\mathbf{x}; q,t) = E_{\lambda}(\mathbf{x}; q, t).$
\end{defn}

We can use the ASEP polynomials to define Macdonald polynomials.
The fact that  \cref{lem:Macdonald} agrees with the original definition
of Macdonald polynomials comes from 
\cite[Lemma~1]{cantini-degier-wheeler-2015}.

\begin{defn}\label{lem:Macdonald}
	Let $\lambda$ be a partition.  We define the \emph{Macdonald 
	polynomial} 
$P_\lambda(\mathbf{x}; q, t)$ by 
 \begin{equation}
\label{sum-asep poly}
P_\lambda(\mathbf{x}; q, t) = \sum_\eta \ASEP_\eta(\mathbf{x}; q, t),
\end{equation}
where the sum runs over all $\eta \in S_{\lambda}$, i.e.
the permutations $\eta$ of $\lambda$. 
\end{defn}

When we specialize the $x_i$'s and $q$ to be $1$,
we obtain a relation between the ASEP polynomials and the multispecies ASEP.

\begin{prop}\cite[Corollary 1]{cantini-degier-wheeler-2015}
Let $\lambda=(\lambda_1,\dots,\lambda_n)$ be a partition.
The steady state probability that the multispecies ASEP is in 
state $\eta\in S_{\lambda}$ is 
\begin{equation*}
	\frac{\ASEP_{\eta}(1,\dots,1; 1,t)}
	{P_{\lambda}(1,\dots,1;1,t)}.
\end{equation*}
\end{prop}

 Macdonald showed~\cite[Section~ VI, Chapter~4]{Macdonald} that
\begin{equation}
\label{Pla-q=1}
	P_\lambda(\mathbf{x};  1, t) = e_{\lambda'}(\mathbf{x}),
\end{equation}
{where $\lambda'$ denotes the conjugate partition of 
$\lambda$, and $e_{\lambda'}(\mathbf{x})$ denotes the corresponding elementary symmetric polynomial.}

Note that by \eqref{Pla-q=1},  $P_\lambda(1, \dots, 1; 1, t)$
is independent of $t$.

\begin{eg}\label{eg:F}
Consider all permutations of the tuple $(0, 1, 2)$ as compositions. The ASEP polynomials for each of them are as follows:
\[
\begin{array}{|c|c|}
\hline
\eta & \ASEP_\eta(x_1, x_2, x_3; q, t) \\
\hline
(0, 1, 2) & x_{2} x_{3}^{2}
+ \frac{t (1- t)}{1-q t^{2}} x_{1} x_{2} x_{3}\\
(0, 2, 1) & x_{2}^{2} x_{3} 
+ \frac{(1- t)}{1-q t^{2}} x_{1} x_{2} x_{3} \\
(1, 0, 2) & x_{1} x_{3}^{2} 
+ \frac{(1- t)}{1-q t^{2}} x_{1} x_{2} x_{3}\\
(1, 2, 0) & x_{1} x_{2}^{2} 
+ \frac{q t (1- t)}{1-q t^{2}} x_{1} x_{2} x_{3} \\
(2, 0, 1) & x_{1}^{2} x_{3} 
+ \frac{q t (1- t)}{1-q t^{2}} x_{1} x_{2} x_{3} \\
(2, 1, 0) & x_{1}^{2} x_{2} 
+ \frac{q(1- t)}{1-q t^{2}} x_{1} x_{2} x_{3} \\
\hline
\end{array}
\]
One can check that $\ASEP_{(2,1,0)} = E_{(2,1,0)}$ from \cref{eg:nsym-macd}.
\end{eg}

\begin{eg}
\label{eg:symm-macd}
The expansions of the Macdonald polynomials for partitions of size $3$ in the variables $x_1, x_2, x_3$ in terms of 
monomial symmetric functions $m_\lambda=m_\lambda(x_1, x_2, x_3)$
are as follows:
\[
\hspace*{-0.7cm}
\begin{array}{|c|c|}
\hline
\lambda & P_\lambda(x_1, x_2, x_3; q, t) \\
\hline
(3,0,0) &  m_{(3)} 
+ \frac{\left(q^2+q+1\right) (t-1)}{q^2 t-1} m_{(2,1)} + \frac{(q+1) \left(q^2+q+1\right) (t-1)^2}{(q t-1) \left(q^2 t-1\right)} m_{(1,1,1)}\\[0.2cm]
\hline
(2, 1,0) & m_{(2,1)} 
+ \frac{(t-1) (2 q t+q+t+2)}{q t^2-1} m_{(1,1,1)}\\[0.2cm]
\hline
(1, 1, 1) & m_{(1,1,1)} \\
\hline
\end{array}
\]
Note that the sum of the ASEP polynomials from \cref{eg:F}
equals $P_{(2,1,0)}$ as given above.
	It turns out that the expansion of the Macdonald polynomials is the same irrespective of the number of variables ({as long as there are enough variables}).	
\end{eg}

The following result will be useful in proving 
\cref{thm:obs}.
\begin{lem}\cite[Theorem 3.4 (18)]{corteel-mandelshtam-williams-2022}
\label{lem:ASEP pol sym}
Let $\eta$ be a composition and 
$1\leq i\leq n$. 
Then 
$\ASEP_\eta + \ASEP_{s_i\eta}$
is symmetric in $x_i$ and $x_{i+1}$.
\end{lem}

\subsection{Permuted basement Macdonald polynomials}
There is a more general class of polynomials that generalize both  the nonsymmetric Macdonald
polynomials and the ASEP polynomials.  These polynomials are called 
\emph{permuted basement Macdonald polynomials}
$E_{\alpha}^{\sigma}({\mathbf x}; q, t)$
(where $\sigma \in S_n$ and $\alpha$ is a composition with $n$ parts);
they were introduced by Ferreira in \cite{ferreira-2011} and further 
 studied in \cite{alexandersson-2019} and 
\cite{alexandersson-sawhney-2019}.
We do not need here the definition, but only the fact that they specialize to 
both the nonsymmetric Macdonald polynomials and the ASEP polynomials.
In particular, the nonsymmetric Macdonald polynomial $E_{\eta}$ 
is equal to $E_{\rev(\eta)}^{w_0}$, where $\rev(\eta)$ denotes the reverse composition $(\eta_n, \eta_{n-1},\dots,\eta_1)$ of $\eta=(\eta_1,\dots,\eta_n)$ and $w_0$ denotes the
longest permutation $(n,\ldots,2,1)$ {in one-line notation}.  
We also have the following.

\begin{prop}\cite[Proposition 4.1]{corteel-mandelshtam-williams-2022}
\label{prop:CMW}
For $\eta=(\eta_1,\ldots,\eta_n)$, 
define $\inc(\eta)$ to be the sorting of the parts of $\eta$ in increasing order. Then
\[F_{\eta}=E_{\inc(\eta)}^{\sigma}
\]
where 
$\sigma$ is the element of $S_n$ with longest  length such that
      $$\eta_{\sigma(1)} \leq \eta_{\sigma(2)} \leq 
        \dots \leq \eta_{\sigma(n)}.$$
\end{prop}

\section{Nonsymmetric Macdonald and ASEP polynomials at $q=1$}
\label{sec:symm-fn}

The main goal of this section is
to prove \cref{prop:f-projection}, which will 
be a main ingredient in our proof of 
\cref{thm:multilrep-rej-statprob}. Along the way we will 
prove various 
properties of nonsymmetric Macdonald polynomials
and ASEP polynomials at the specialization $q=1$.

\subsection{Nonsymmetric Macdonald polynomials at $q=1$}\label{sec:project}

In this section we show that certain ratios of nonsymmetic Macdonald
polynomials at $q=1$ are elementary symmetric polynomials, building on 
work of
\cite{alexandersson-sawhney-2019}. 
Our first main result is the following.
\begin{thm}\label{thm:main}
	Consider a composition $\eta=(\eta_1,\dots,\eta_n)$ such that its parts of sizes $a$ and $b$ (with $a<b$) occur in 
	increasing order from left to right in $(\eta_1,\dots,\eta_n)$, and such that $\eta$ has no parts of size
	$h$ for any $a<h<b$.  (In this case, we say that the classes $a$ and $b$ are \emph{adjacent}.)
	Let $\hat{\eta}$ be the composition obtained from $\eta$ by changing all $a$'s to $b$, or changing 
	all $b$'s to $a$.  Then at $q=1$,
	\begin{equation}
		\frac{E_{\eta}}{E_{\hat{\eta}}} = 
		\left( 
		\frac{e_{\#\{i  \vert  \eta_{i} \geq b\}}}
		{e_{\#\{i  \vert  \hat{\eta}_{i} \geq b\}}} \right)^{b-a}.
	\end{equation}
\end{thm}

\begin{eg}
At $q=1$, we have 
\begin{align*}
	\frac{E_{2515}}{E_{2212}} &= \left(\frac{e_2}{e_0} \right)^{3}, \qquad 
	&\frac{E_{2515}}{E_{5515}} &= \left(\frac{e_2}{e_3} \right)^{3}, \\
	\frac{E_{201}}{E_{200}} &= \left(\frac{e_2}{e_1} \right)^{1}, \qquad
	&\frac{E_{3311220}}{E_{3311110}} &= \left(\frac{e_4}{e_2} \right)^{1}.
	\end{align*}
\end{eg}

\begin{rem}
Some results from 
	\cite{alexandersson-sawhney-2019} 
	will be useful to us.  However, we have to be careful of conventions.
	Sage and \cite{corteel-mandelshtam-williams-2022} have the same conventions, but those conventions are different from \cite{alexandersson-sawhney-2019}; in particular,
the compositions indexing Macdonald polynomials are reversed in these references.
That is, the polynomial called 
$E_{(\eta_1,\dots,\eta_n)}$ 
in \cite{alexandersson-sawhney-2019} is the same 
	as the polynomial called $E_{(\eta_n,\dots,\eta_1)}$ in Sage and \cite{corteel-mandelshtam-williams-2022}.
\end{rem}

\begin{defn}\label{def:ws}
The \emph{weak standardization} of a composition
$\eta$, denoted $\tilde{\eta}$, is the lexicographically
	smallest composition with the property that if $\eta_i < \eta_j$, then $\tilde{\eta}_i < \tilde{\eta}_j$ for all pairs $i,j$.  {The \emph{conjugate} $\eta'$ of a
	composition $(\eta_1,\dots,\eta_n)$ is obtained by 
	drawing a left-justified diagram consisting of rows of lengths
	$(\eta_1,\dots,\eta_n)$, then reading the columns from left to right
	and recording the number of boxes in each column.}
\end{defn}

For example, if 
$\eta = (2, 5, 1, 4, 4, 5, 4)$, then $\tilde{\eta} = (1, 3, 0, 2, 2, 3, 2)$, and 
$\eta' = (7,6,5,5,2).$
The theorem below 
is due to \cite{alexandersson-sawhney-2019}, but 
we have phrased it using the conventions of Sage and \cite{corteel-mandelshtam-williams-2022}.

\begin{thm}\cite[Equation (2) and Theorem 18]{alexandersson-sawhney-2019}
\label{cor:AS0}
	Choose a basement $\sigma\in S_n$ and composition $\eta=(\eta_1,\dots,\eta_n)$, and let $\tau=\tilde{\eta}$ denote
the weak standardization of $\eta$. Then
\begin{equation}\label{eq:useful}
E_{\eta}^{\sigma}(\mathbf{x}; 1, t) = 
	\left( \frac{e_{\eta'}(\mathbf{x})}{e_{\tau'}(\mathbf{x})} \right)
	E_{{\tau}}^{\sigma}(\mathbf{x}; 1, t),
\end{equation}
where $\eta'$ denotes the composition which is conjugate to $\eta$, and 
	$e_{\eta'}(\mathbf{x})/e_{\tau'}(\mathbf{x})$ 
	is an elementary symmetric polynomial independent of $t$.
And if the composition $\eta$ has weakly increasing parts, then 
$$E_{\eta}^{\sigma}(\mathbf{x}; 1, t) = e_{\eta'}(\mathbf{x}).$$  
\end{thm}

Note that if $\mu$ is a partition and $\eta \in S_{\mu}$,
then $\eta' \in S_{\mu'}$, that is, both $\eta'$ and $\mu'$
have the same set of parts.  
Using \cref{cor:AS0} and \cref{prop:CMW}, we now obtain the following.

\begin{cor}
\label{cor:AS}
Let $\mu=(\mu_1,\dots,\mu_n)$ be a 
partition, and let $\eta=(\eta_1,\dots,\eta_n)\in S_{\mu}$. Then we have
\[
\frac{\ASEP_\eta(\mathbf{x}; 1, t)}{\ASEP_{\tilde{\eta}}(\mathbf{x}; 1, t)}
	= h_{\mu}(\mathbf{x}),
\]
	where $h_{\mu}(\mathbf{x}):=e_{\mu'}(\mathbf{x})/e_{\tilde{\mu}'}(\mathbf{x})$ is a (symmetric) polynomial in $\mathbf{x} = (x_1, \dots, x_n)$ which is independent of $\eta$ and of $t$.

And if the composition $\eta$ has weakly increasing parts, then 
the nonsymmetric Macdonald polynomial $E_{\eta}$ satisfies
$$E_{\eta}(\mathbf{x}; 1, t) = e_{\eta'}(\mathbf{x}).$$  
\end{cor}

\begin{prop}\cite[Proposition 16]{alexandersson-sawhney-2019}\label{prop:AS}
If $\eta$ is a composition,
then 
	\begin{equation*}
	\frac{E_{\eta}(\mathbf{x}; 1, t)}
	{E_{\tilde{\eta}}(\mathbf{x}; 1, t)} 
	\end{equation*}
	 is an elementary symmetric polynomial.
\end{prop}

Before proving \cref{thm:main}, we recall the \emph{shape permuting
operator} from \cite[Equation (17)]{HHL3}.

\begin{prop} \cite{HHL3}
\label{prop:shape}
Let $\nu$ be a composition, and suppose $\nu_i>\nu_{i+1}$. 
Write 
\[
r_i(\nu)=
\#\{j<i \mid \nu_{i+1}<\nu_{j}\leq \nu_i\}
+
\#\{j>i \mid \nu_{i+1}\leq \nu_j < \nu_i\}.
\]
Then
\begin{equation}
\label{shape-permute}
E_{s_i \nu}(\mathbf{x};q,t) = 
\left(T_i + \frac{1-t}{1-q^{\nu_{i+1}-\nu_i}t^{r_i(\nu)}}\right)
E_{\nu}(\mathbf{x};q,t).
\end{equation}
\end{prop}
(In the terminology of \cite{HHL3}, 
the exponent $\nu_{i+1}-\nu_i$ of $q$ 
appearing in \cref{shape-permute}
is $1+\leg(u)$, 
and the exponent $r_i(\nu)$ of $t$ is $\arm(u)$,
where $u$ is the box $(i, \nu_{i+1}+1)$ in the
column diagram of the composition $\nu$.)

\begin{proof}[Proof of \cref{thm:main}]
We start by reducing the proof of \cref{thm:main} to the 
case of a composition whose parts are weakly increasing.
Note that by our assumptions on $\eta$ and $\hat{\eta}$,
we will have $\eta_i>\eta_{i+1}$ if and only if 
$\hat{\eta}_i > \hat{\eta}_{i+1}$. In this case, 
it is not hard to see that $r_i(\eta)=r_i(\hat{\eta})$.
Thus if we set $q=1$, then the shape permuting operator
is the ``same'' for both $\eta$ and $\hat{\eta}$, that is, 
$$E_{s_i \eta}(\mathbf{x};1,t) = 
\left(T_i + \frac{1-t}{1-t^{r}}\right)
E_{\eta}(\mathbf{x};1,t)$$ and 
	$$E_{s_i \hat{\eta}}(\mathbf{x};1,t) = 
\left(T_i + \frac{1-t}{1-t^{r}}\right)
	E_{\hat{\eta}}(\mathbf{x};1,t),$$ 
	where $r=r_i(\eta)=r_i(\hat{\eta})$.
By applying the same sequence of 
	shape permuting operators at $q=1$ to $E_{\eta}(\mathbf{x};1,t)$ and 
$\left( 
\frac{e_{\#\{i  \vert  \eta_{i} \geq b\}}}
	{e_{\#\{i \vert  \hat{\eta}_{i} \geq b\}}} \right)^{b-a}
	E_{\hat{\eta}}(\mathbf{x};1,t)$,  and using \cref{rem:operator},
we can therefore 
reduce to the case 
 of a composition whose parts are weakly increasing.

We now suppose that $\eta$ has weakly increasing parts; 
likewise $\hat{\eta}$ has weakly increasing parts.
 We  now apply \cref{cor:AS}, which says that 
at $q=1$, 
$$\frac{E_{\eta}}{E_{\hat{\eta}}} = \frac{e_{\eta'}}{e_{\hat{\eta}'}}.$$
We have that $\frac{e_{\eta'}}{e_{\hat{\eta}'}} = 
\left( 
\frac{e_{\#\{i \vert  \eta_{i} \geq b\}}}
	{e_{\#\{i  \vert  \hat{\eta}_{i} \geq b\}}} \right)^{b-a},$
	so we are done.
\end{proof}

Now what we are really interested in is quantities likes 
$E_{201}/E_{100}$ or $E_{3311220}/E_{2211110}$.
The denominator is obtained in two stages from the numerator: by merging two adjacent classes $a<b$ appearing in increasing order as in \cref{thm:main},
and then applying weak standardization, as in 
	\cref{prop:AS}.  By applying \cref{thm:main} and \cref{prop:AS}, we obtain the following.

\begin{cor}\label{cor:main}
	Consider a composition $\eta=(\eta_1,\dots,\eta_n)$ such that its parts of sizes $a$ and $b$ (with $a<b$) occur in 
increasing order from left to right in $(\eta_1,\dots,\eta_n)$, and such that $\eta$ has no parts of size
	$h$ for any $a<h<b$.  
	Let $\hat{\eta}$ be the composition obtained from $\eta$ by changing all $a$'s to $b$, or changing 
	all $b$'s to $a$.  Let $\tilde{\hat{\eta}}$ be the composition obtained from $\hat{\eta}$ by applying
	the weak standardization (i.e. we find the smallest composition whose parts are in the same relative order
	as those of $\hat{\eta}$).  Then at $q=1$,
		$E_{\eta}/E_{\tilde{\hat{\eta}}}$ is a ratio of elementary symmetric polynomials. 
\end{cor}

\subsection{Projections of ASEP polynomials at $q=1$}

In this section we explain some properties of ASEP polynomials that hold
true at $q=1$.  
These properties are algebraic analogues of the recolouring 
properties of the 
$t$-PushTASEP
 discussed in \cref{sec:projections}, in particular
\cref{prop:colouring-stationary}.

Recall the definition of weakly order-preserving
from \cref{def:orderpreserving}.
Throughout this section we will fix the following notation.

\begin{notation}\label{not:proj}
Let $\phi:\NN\to \NN$ be a weakly order-preserving function.
Let $\mu=(\mu_1,\dots,\mu_n)$ and 
$\lambda=(\lambda_1,\dots,\lambda_n)$ be partitions with $\phi(\lambda)=\mu$. 
For $\eta\in S_{\mu}$, let
	\begin{equation}\label{eq:G}
		G_\eta = G_\eta(\mathbf{x};t) := 
	\sum_{\substack{\zeta\in S_{\lambda} \\ \phi(\zeta)=\eta}}
		F_{\zeta}(\mathbf{x}; 1, t).
	\end{equation}
\end{notation}

\begin{eg}\label{eg:project}
Suppose that $\phi(0)=\phi(1)=1$ and $\phi(j)=j$ for $j\notin \{0,1\}$.
	Let $\lambda=(2,1,0)$ and $\mu = \phi(\lambda) = (2,1,1)$.
	Then  
	{
	\begin{align*}
		G_{(2,1,1)} & = F_{(2,1,0)}(\mathbf{x};1,t) + F_{(2,0,1)}(\mathbf{x}; 1, t) = x_1(x_1 x_2 + x_1 x_3 + x_2 x_3), \\
		G_{(1,2,1)} & = F_{(1,2,0)}(\mathbf{x}; 1, t) + F_{(0,2,1)}(\mathbf{x}; 1, t) = x_2(x_1 x_2 + x_1 x_3 + x_2 x_3), \\
		G_{(1,1,2)} & = F_{(1,0,2)}(\mathbf{x}; 1, t) + F_{(0,1,2)}(\mathbf{x}; 1, t) = x_3(x_1 x_2 + x_1 x_3 + x_2 x_3),
	\end{align*}
where the values of the ASEP polynomials are taken from \cref{eg:210 probs}.}
\end{eg}

\begin{prop}\label{lem:g-is-KZ} 
Use  \cref{not:proj}.
	Then the family $\{G_\eta \ \vert \ \eta\in S_{\mu}\}$ is a KZ family, i.e. 
it satisfies the relations of \cref{qKZ} at $q=1$.
\end{prop}
\begin{proof}
It suffices to prove \cref{lem:g-is-KZ} when $\phi$ has a particular form:
namely, there is a natural number $\ell$ such that 
$\phi(\ell)=\phi(\ell+1) = \ell+1$, and $\phi(j)=j$ for $j\notin \{\ell,\ell+1\}$.
Suppose that the components of $\lambda$ contain $a$ instances of $\ell$ and 
$b$ instances of $\ell+1$.  Then each $\eta \in S_{\mu}$
contains $a+b$ instances of $\ell+1$, and 
$G_{\eta}$ is a sum of the ${a+b \choose a}$ polynomials $F_{\zeta}$,
where $\zeta$ is obtained from  $\eta$ by changing $a$ of the 
$\ell+1$'s in $\eta$ to $\ell$.

Suppose that $1\leq i \leq n-1$ is a position such that 
at most one of $\eta_i$ and $\eta_{i+1}$ equals $\ell+1$.
Then each $\zeta$ occurring in the right-hand-side of 
\eqref{eq:G} 
	has the property that $\zeta_i$ and $\zeta_{i+1}$
have the same relative order as $\eta_i$ and $\eta_{i+1}$.
Without loss of generality say  that $\eta_i>\eta_{i+1}$.
	(The arguments in the other cases are similar.)
	Then we have that \begin{align*}
		T_i G_{\eta} & =
	\sum_{\substack{\zeta\in S_{\lambda} \\ \phi(\zeta)=\eta}}
		T_i (F_{\zeta}(\mathbf{x}; 1, t)) \\
     &=
	\sum_{\substack{\zeta\in S_{\lambda} \\ \phi(\zeta)=\eta}}
		 F_{s_i(\zeta)}(\mathbf{x}; 1, t) \\
     &= G_{s_i(\eta)},
	\end{align*}
	as desired, 
where we used the fact
(cf \cref{def:ASEPpolynomials}) that the polynomials 
$F_{\zeta}(\mathbf{x}; q, t)$ are themselves a qKZ family.

Now suppose that $\eta_i = \eta_{i+1} = \ell+1$.
Then the $\zeta$ appearing on the right-hand side of 
\eqref{eq:G} satisfy either 
$\zeta_i=\zeta_{i+1}$, or 
$\zeta_i = \ell+1$ and $\zeta_{i+1} = \ell$,
or 
$\zeta_i = \ell$ and $\zeta_{i+1} = \ell+1$,
so we will divide up the sum accordingly.
In what follows, we will abbreviate 
$F_{\zeta}(\mathbf{x}; 1, t)$ by 
$F_{\zeta}$, and omit the 
conditions 
that $\zeta\in S_{\lambda}$ and $\phi(\zeta)=\eta$ in all 
sums below.
So we get 
 \begin{align*}
	T_i G_{\eta} 
	 &=	 \sum_{\substack{\zeta_i=\zeta_{i+1}}}
		T_i (F_{\zeta}) +
	 \sum_{\substack{\zeta_i=\ell+1, \zeta_{i+1} = \ell}}
		T_i (F_{\zeta})  + 
	 \sum_{\substack{\zeta_i=\ell, \zeta_{i+1} = \ell+1}}
		T_i (F_{\zeta}).\\
	 &=	 \sum_{\substack{\zeta_i=\zeta_{i+1}}}
		 t F_{\zeta} +
	 \sum_{\substack{\zeta_i=\ell, \zeta_{i+1} = \ell+1}}
		F_{\zeta}  + 
	 \sum_{\substack{\zeta_i=\ell, \zeta_{i+1} = \ell+1}}
		\left( (t-1) F_{\zeta} + tF_{s_i \zeta}\right).\\
	 &=	 \sum_{\substack{\zeta_i=\zeta_{i+1}}}
		 t F_{\zeta} +
	 \sum_{\substack{\zeta_i=\ell, \zeta_{i+1} = \ell+1}}
		\left(t F_{\zeta} + tF_{s_i \zeta}\right).\\
     &= t G_{\eta},
\end{align*}
as desired.

Finally since the polynomials $F_{\zeta}$ satisfy \eqref{thirdproperty},
at $q=1$ we get 
$$F_{\zeta}(\mathbf{x}; 1,t)
	       =F_{(\zeta_n,\zeta_1,\dots,\zeta_{n-1})}(x_n,x_1,\dots,x_{n-1}; 1,t).$$
But then we have
\begin{align*}
	G_\eta(\mathbf{x};t) 
	&=	\sum_{\substack{\zeta\in S_{\lambda} \\ \phi(\zeta)=\eta}}
	F_{\zeta}(\mathbf{x}; 1, t) \\
	&=	\sum_{\substack{\zeta\in S_{\lambda} \\ \phi(\zeta)=\eta}}
	F_{(\zeta_n, \zeta_1,\dots,\zeta_{n-1})}(x_n,x_1,\dots,x_{n-1}; 1, t) \\
	&= G_{(\eta_n,\eta_1,\dots,\eta_{n-1})}(x_n,x_1,\dots,x_{n-1},t),
\end{align*}
as desired.
\end{proof}

\begin{lem}\label{twoKZ}
	Suppose that  
$\{f_\eta \ \vert \  \eta\in S_{\lambda}\}$ and $\{g_\eta \ \vert \ 
\eta\in S_{\lambda}\}$ 
are both KZ families. If the ratio $h:=g_\lambda/f_\lambda$
of the partition-indexed terms in the two families 
is symmetric in $x_1,\dots, x_n$, 
then the families $(f_\eta)$ and $(g_\eta)$ are proportional 
to each other, i.e.\ $g_\eta/f_\eta=h$ for all $\eta\in S_\lambda$. 
\end{lem}

\begin{proof}
Let $\lambda$ be a partition and let $\eta\in S_\lambda$. 
We can obtain $\eta$ from $\lambda$ by a sequence of nearest-neighbour transpositions,
each of which changes a pair of entries from decreasing order into 
increasing order. Hence from \eqref{firstproperty} we have
$f_\eta = T_{i_k}\dots T_{i_1} f_\lambda$ 
and $g_\eta = T_{i_k}\dots T_{i_1} g_\lambda$, for some sequence
$i_1,\dots i_k$. Then by induction and \cref{rem:operator}, if 
$g_{\lambda} = h f_\lambda$ where $h$ is symmetric in $x_1,\dots, x_n$,
then $g_\eta = h f_\eta$. 
\end{proof}

\begin{defn}\label{def:inc}
Use \cref{not:proj}.
Let $\inc(\lambda, \phi)$ be the 
lexicographically smallest $\zeta\in S_\lambda$
such that $\phi(\zeta)=\mu$.
\end{defn}
\begin{eg}
Suppose 
$\phi(6)=\phi(5)=\phi(4)=5, \phi(3)=4, \phi(2)=\phi(1)=1$.
	Let $\lambda=(6,6,6,5,4,3,3,2,2,1,1)$, so that 
	$\mu=(5,5,5,5,5,4,4,1,1,1,1)$. 
	Then $\inc(\lambda,\phi)=(4,5,6,6,6,3,3,1,1,2,2)$.
\end{eg}

\begin{prop}
\label{prop:Gmu}
Use  \cref{not:proj}.
	We have that $G_\mu = E_{\inc(\lambda,\phi)}(q=1)$.
\end{prop}

\begin{proof} 
The ASEP polynomials and the nonsymmetric Macdonald polynomials
are related via a triangular change of basis 
\cite[(23)]{cantini-degier-wheeler-2015}, and hence
the span of $\{E_\eta(q=1) \ \vert \ \eta\in S_{\lambda}\}$ 
is the same as the span of 
$\{F_\eta(q=1) \ \vert \  \eta\in S_{\lambda}\}$.
Therefore  $G_\mu$ lies  in 
the span of $\{E_\eta(q=1) \ \vert \ \eta\in S_{\lambda}\}$. 

We know from \cref{lem:g-is-KZ}  that 
$\{G_{\eta} \ \vert \ \eta\in S_{\mu}\}$ is a KZ family.
Proceeding as in the proof of 
\cite[Lemma 1.23]{corteel-mandelshtam-williams-2022}, 
we can 
use the relations of the KZ family to show that $G_{\mu}$ is an eigenvector
of each $Y_i$, and that the eigenvalue of $Y_i$ on $G_\mu$
is the same as the eigenvalue of $Y_i$ on  
$F_{\mu} = E_{\mu}$ 
when $q=1$.
But now it is easy to see from 
\eqref{eq:eigenvalue} that the 
eigenvalues of $Y_i$ on $E_{\mu}$ at $q=1$
are the same as the eigenvalues of $Y_i$ on  
$E_{\inc(\lambda,\phi)}$ at $q=1$.
Recall from \cref{rem:eigenvalues}
that the polynomials $\{E_\eta(q=1), \eta\in S_\lambda$\}
have distinct tuples of eigenvalues at $q=1$. Since $G_\mu$
is in the span of those polynomials and its eigenvalues
agree with those of $E_{\inc(\lambda, \phi)}$, it must be a multiple
of $E_{\inc(\lambda,\phi)}$. Since both polynomials
have
the same leading term, it follows
that in fact $G_\mu = E_{\inc(\lambda,\phi)}$.
\end{proof}

\begin{thm}\label{thm:symm}
Use \cref{not:proj}.
We have that at $q=1$, $E_{\inc(\lambda, \phi)}$ is a symmetric function
multiple of $E_{\mu} = F_\mu$. 
\end{thm}
\begin{proof}
Since $\mu$ is a partition, we have that $F_{\mu}=E_{\mu}$
from the definition of ASEP polynomials.
Now this theorem can be obtained by repeated applications
of \cref{cor:main},
 together with \cref{prop:AS}.
\end{proof}

\begin{thm}\label{prop:f-projection}
	Use  \cref{not:proj}.
For all $\eta\in S_{\mu}$
and for $q=1$, we have
\[
\frac{G_\eta}{P_\lambda} = 
	\frac{G_\eta}{\sum_{\theta \in S_{\mu}} G_{\theta}} = 
	\frac{F_\eta}{\sum_{\theta \in S_{\mu}} F_{\theta}} = 
	\frac{F_\eta}{P_\mu }.
\]
\end{thm}
\begin{proof}
By \cref{prop:Gmu},
	we have that $G_\mu = E_{\inc(\lambda,\phi)}(q=1)$.
Since $\mu$ is a partition, we also know by 
\cref{def:ASEPpolynomials}
	that $F_{\mu} = E_{\mu}$.
So by 
\cref{thm:symm}, $h:=G_{\mu}/F_{\mu}(q=1)$ is 
symmetric in $x_1,\dots,x_n$. We can therefore apply 
\cref{twoKZ} to conclude that 
	$G_{\eta} = h F_{\eta}$ for all $\eta \in S_{\mu}$.
	But now it follows that 
	$$\frac{G_\eta}{\sum_{\theta \in S_{\mu}} G_{\theta}} = 
	\frac{F_\eta (q=1)}{\sum_{\theta \in S_{\mu}} F_{\theta}(q=1)}.$$
The other equalities in the proposition follow after using the fact that 
	$\sum_{\theta\in S_{\mu}} F_{\theta} = P_{\mu}$ 
	and $\sum_{\theta \in S_{\mu}} G_{\theta} = 
	\sum_{\xi \in S_{\lambda}} F_{\xi}(q=1) = P_{\lambda}(q=1).$
\end{proof}

\begin{eg}
We continue \cref{eg:project}.  Let $\eta = (1,2,1)$.
Then  by 
\cref{prop:f-projection},  when we specialize to 
	$q=1$, we have 
\begin{align*}
\frac{G_{(1,2,1)}}{P_{(2,1,0)}} &=
\frac{G_{(1,2,1)}}{G_{(2,1,1)}+G_{(1,2,1)}+G_{(1,1,2)}}\\
&= 
	\frac{F_{(0,2,1)}+F_{(1,2,0)}}{F_{(2,0,1)}+F_{(2,1,0)} + F_{(0,2,1)}+F_{(1,2,0)}+F_{(0,1,2)}+F_{(1,0,2)}} \\
&=
\frac{F_{(1,2,1)}}{F_{(2,1,1)}+F_{(1,2,1)}+F_{(1,1,2)}} 
	= \frac{F_{(1,2,1)}}{P_{(2,1,1)}}
	{= \frac{x_2}{x_1 + x_2 + x_3}}.
\end{align*}
\end{eg}

\section{Multiline diagrams}
\label{sec:asep}
In this section we define \emph{multiline diagrams}.
These are combinatorial objects which were introduced 
in the context of the multispecies TASEP in \cite{ferrari-martin-2009}, and have subsequently been generalised to a range of
related settings, including in \cite{martin-2020}
in the context of the multispecies ASEP, and 
in \cite{corteel-mandelshtam-williams-2022} in 
the context of Macdonald polynomials. 

\subsection{Definition of multiline diagrams}
The definition we present here is a special case of the 
definition given in 
\cite{corteel-mandelshtam-williams-2022}: 
we impose $q=1$, and only define multiline diagrams with bottom row
$\eta$, where $\eta=(\eta_1,\dots,\eta_n)$ is a composition whose nonzero parts
are distinct.

Fix a partition $\lambda=(\lambda_1,\dots,\lambda_n) = 
 \langle s^{m_s}, \dots, 1^{m_1}, 0^{m_0} \rangle$, 
 where $\lambda_1=s$ and 
$m_{i}\in\{0,1\}$ for all $i\geq 1$. 
Let $a_i=\sum_{r=i}^s m_r$
be the number of particles of type $i$ or higher.

\begin{defn}
	A \emph{ball system} with \emph{content $\lambda$}
is an array with $s$ rows and $n$ columns 
in which each of the $sn$ positions
is either empty or occupied by a ball, and in which
the number of balls in row $i$ is $a_i$. 
We number the rows
from $1$ to $s$ from bottom to top, and the sites 
from $1$ to $n$ from left to right.
\end{defn}
 
\begin{defn}
\label{defn:md}
Given a ball system $B$ with content $\lambda$, 
a \emph{multiline diagram} $D$  (with pattern $B$) is an assignment of types (positive integers)
to the balls of $B$. The $a_i$ balls in row $i$ 
are given distinct labels from the set (of size $a_i$)
$\{r\geq i: m_r=1\}$. 
The labelling satisfies the following 
constraint: if two vertically adjacent sites $(i,j)$ and $(i, j+1)$ 
both contain a ball, then the label of the lower site $(i,j)$ 
must be at least as large as the label of the upper site $(i,j+1)$. 
	We also call $\lambda$ the content of the diagram. 
\end{defn}

Each row $i$ of a multiline diagram $D$ gives rise
to a composition $\rho^{(i)}=(\rho^{(i)}_1,\dots,\rho^{(i)}_n)$, where
\[
\rho^{(i)}_j=
\begin{cases}
h&\text{if $D$ has a ball with label $h$ at $(i,j)$};\\
0&\text{if $(i,j)$ is empty in $D$}.
\end{cases}
\]
The composition $\rho^{(1)}$ associated 
to the bottom row of $D$ is sometimes called the
\emph{type} of $D$. 
If $D$ has content $\lambda$ then $\rho^{(1)}(D)\in S_\lambda$.

\subsection{Generating multiline diagrams}
\label{sec:generate}
We next explain a procedure for randomly generating 
a multiline diagram {with} content $\lambda$. As
above let $s$ be the largest entry of $\lambda$, 
so that the diagram has $s$ rows.

We first generate the ball system. The occupancies on 
different rows are independent. On row $r$, where we
require $a_r$ balls, we occupy a given set of sites $A_r\subset[n]$
of size $a_r$ with probability proportional to 
$\prod_{j\in A_r} x_j$. (This is precisely the stationary 
distribution of the single-type $t$-PushTASEP, as given  in 
\cref{prop:ss-singlePushTASEP}
 -- the normalising constant is given 
by the elementary symmetric function $e_{a_r}(x_1, \dots, x_n)$.)

Now we assign labels to the balls in the system. This is 
done recursively line by line, working 
from top (row $s$) to bottom (row $1$). 

\begin{itemize}
\item
We assign the single ball on row $s$ (the top row) the label $s$. 
\item
Now suppose we have already labelled the rows from $s$
down to $r+1$,
and it is time to label row $r$.
We consider the balls in row $r+1$ one by one in 
decreasing order of 
their label. We match each one to a ball on row $r$, 
and that ball on row $r$ will be given the same label.

Suppose we are considering the ball with label $h$ on row $r+1$, 
with position in some column $j$. 
First we check whether the site immediately 
below it, $(r,j)$, has a ball which has not yet been labelled.
If so we match to that ball, labelling it $h$. This is called 
a \textit{trivial match}.
Otherwise, 
consider all the balls remaining in row $r$ which have not yet been labelled -- there are a total of $K:=a_r - a_{h+1}$ of them. Suppose their columns, listed from left to right in cyclic order starting from column $a$, are $j_1, j_2, \dots, j_K$: that is, 
\[
0<(j_1-j) \pmod n < (j_2-j) \pmod n < \dots < (j_K-j) \pmod n.
\]
Now we match the $h$-labelled ball at $(r+1,j)$ 
	to the ball at position $(r, j_k)$ 
with probability $t^{k-1}/(1+t+\dots+t^{K-1})$,
and assign the label $h$ to that ball. 

\item
In this way every ball on row $r+1$ gets matched to a ball on row $r$. If $a_r=a_{r+1}$ then we have labelled every ball on row $r$,
and we move on to labelling the balls in the rows 
		below. If instead $a_r=a_{r+1}+1$, 
then there remains a single unlabelled ball on row $r$, and we 
assign it label $r$. 
\end{itemize}
Proceeding in this way we construct a labelling having
the properties in \cref{defn:md}.

\subsection{Weight function for multiline diagrams}
Closely related to the above sampling procedure, we define
a weight function on multiline diagrams. For a given collection
of particle counts, and given parameters $(x_1, \dots, x_n)$
and $t$, the probability of sampling a given diagram
using the procedure above is proportional to its \emph{weight},
as defined in \cref{def:weight}. 

\begin{defn}\label{def:weight}
Let $D$ be a multiline diagram with pattern $B$, where $B$ is a 
$s \times n$ ball system.  

For $1\leq j \leq n$ let $c_j$ be the number of balls in column $j$. 
Then the $x$-weight of the diagram is defined by $\wt_x(D)=\prod_{j=1}^n x_j^{c_j}$.  

The $t$-weight is defined as follows. 
Consider $h\geq 2$ such that $m_h=1$. Then there is one ball with label
$h$ in each of the rows $h$ and below. For each $r=1,\dots, h-1$
	we associate a local weight $w_D(h,r)$ to the ball with  
label $h$ in row $r$ as follows:
\begin{itemize}
\item
If the balls of label $h$ in rows $r+1$ and $r$ are in the same column 
(corresponding to a trivial match),
then $w_D(h,r)=1$. 
\item
Otherwise:
\begin{itemize}
\item
Let $K$ be the number of balls in row $r$ with label at most $h$;
\item
Let $j$ be the column with the ball labelled $h$ in row $r+1$, and $j'$ the column with the ball labelled $h$ in row $r$. Consider the interval with left endpoint $j$ and right endpoint $j'$ (wrapping cyclically around the ring if necessary). Let $\ell$ be the number of balls in row $r$ between columns $j$ and $j'$ with label less than $h$.
 We have $0\leq \ell\leq K-1$. 
\end{itemize}
Then define 
\begin{equation}\label{wDhdef}
w_D(h,r)=\frac{t^\ell}{1+t+\dots+t^{K-1}}. 
\end{equation}
\end{itemize}

The $t$-weight of the diagram is then the product of all
{these $w_D(h,r)$ weights}:
\begin{equation}\label{tweight}
\wt_t(D)=\prod_{r=1}^{s-1}
\prod_{\substack{r<h\leq s :\\m_h=1}}  
w_D(h,r).
\end{equation}
Finally we define the \emph{weight} $\wt(D)$ 
of  diagram $D$ to be the product of its \textit{$x$-weight}
and	\emph{$t$-weight}, that is,
$$\wt(D)=\wt_x(D)\wt_t(D).$$
\end{defn}

Note first that $\wt_x(D)$ is proportional to the probability of 
generating the ball system $B$ of $D$ in the first step of the 
procedure above. (The constant of proportionality is 
$\prod_{r=1}^s e_{a_r}(x_1,\dots, x_n)$.)

Also note that for each $h$ and $r$, $w_D(h,r)$ is precisely the probability of making the given matching of the $h$-labelled ball between 
rows $r+1$ and $r$ at the relevant step of the labelling process. 

As a result, the conditional probability of obtaining the configuration
of $D$ on row $r$, given the ball system and the configuration of $D$ on rows $s$ down to 
$r+1$, is the product
\begin{equation}\label{row probability}
\prod_{\substack{h:r<h\leq s :\\m_h=1}}  w_D(h,r),
\end{equation}
which depends on $D$ only through its rows $r$ and $r+1$. 

Hence the probability of generating a given diagram $D$ is proportional 
to $\wt_x(D)$ multiplied by the product of \eqref{row probability}
over $r$ from $s-1$ down to $1$. This yields exactly $\wt(D)$, as required.
See \cref{fig:multiline diagram example} for 
an example of a multiline diagram and its weight function.

\begin{center}
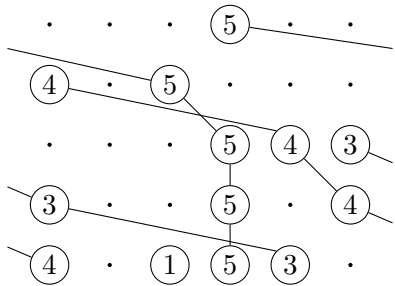
\begin{figure}[h!]

\begin{tikzpicture}[scale=0.8]
\pgfdeclarelayer{background}
\pgfdeclarelayer{foreground}
\pgfsetlayers{background,main,foreground} 

\begin{pgfonlayer}{background}
    \newcommand\rows{5}
    \newcommand\cols{6}

    % Drawing the grid with small dots
    \foreach \row in {1,...,\rows} {
        \foreach \col in {1,...,\cols} {
            \fill (\col,\row) circle (1pt); % Small filled circle
        }
    }
\end{pgfonlayer}

% Placing named balls with numbers on the foreground
\node[draw,circle,fill=white, inner sep=2pt] at (4,5) (ball5-5) {5};
\node[draw,circle,fill=white, inner sep=2pt] at (3,4) (ball5-4) {5};
\node[draw,circle,fill=white, inner sep=2pt] at (4,3) (ball5-3) {5};
\node[draw,circle,fill=white, inner sep=2pt] at (4,2) (ball5-2) {5};
\node[draw,circle,fill=white, inner sep=2pt] at (4,1) (ball5-1) {5};

\node[draw,circle,fill=white, inner sep=2pt] at (1,4) (ball4-4) {4};
\node[draw,circle,fill=white, inner sep=2pt] at (5,3) (ball4-3) {4};
\node[draw,circle,fill=white, inner sep=2pt] at (6,2) (ball4-2) {4};
\node[draw,circle,fill=white, inner sep=2pt] at (1,1) (ball4-1) {4};

\node[draw,circle,fill=white, inner sep=2pt] at (6,3) (ball3-3) {3};
\node[draw,circle,fill=white, inner sep=2pt] at (1,2) (ball3-2) {3};
\node[draw,circle,fill=white, inner sep=2pt] at (5,1) (ball3-1) {3};

\node[draw,circle,fill=white, inner sep=2pt] at (3,1) (ball1-1) {1};

% Drawing lines between named balls (wrapping around ring as needed)
\draw (ball5-5) -- (6.7, 4.6);
\draw (0.3,4.6) -- (ball5-4);
\draw (ball5-4) -- (ball5-3);
\draw (ball5-3) -- (ball5-2);
\draw (ball5-2) -- (ball5-1);

\draw (ball4-4) -- (ball4-3.135);
\draw (ball4-3) -- (ball4-2);
\draw (ball4-2) -- (6.7, 1.7);
\draw (0.3, 1.3) --(ball4-1);

\draw (ball3-3) -- (6.7, 2.7);
\draw (0.3, 2.3) -- (ball3-2);
\draw (ball3-2) -- (ball3-1.135);

\end{tikzpicture}

\caption{A multiline diagram $D$ with 
$n=6$ columns and $s=5$ rows, with
content $\lambda=(5,4,3,1,0,0)$ and bottom 
row $\rho^{(1)}(D)=(4,0,1,5,3,0)\in S_\lambda$. It has 
weight
$\wt(D)=\wt_x(D) \wt_t(D)=x_1^3 x_3^2 x_4^4 x_5^2 x_6^2\, t^2$.
\label{fig:multiline diagram example}}
\end{figure}
\end{center}

\begin{lem}
\label{lem:extra-row}
Let $\lambda=(\lambda_1,\dots,\lambda_n) = 
 \langle s^{m_s}, \dots, 1^{m_1}, 0^{m_0} \rangle$
	be a partition with 
distinct entries, no entry equal to $1$, and exactly one entry equal to $0$. 
	That is, $m_i\in \{0,1\}$, $m_1 = 0$, and $m_0=1$.
Let $\eta, \eta'\in S_\lambda$, and $j\in\{1,\dots, n\}$. 
Then the following quantities are equal.
\begin{itemize}
\item
the probability  
in the $t$-PushTASEP of transitioning from state $\eta$ to state $\eta'$, 
when a bell rings at site $j$ in
		state $\eta$;

\item given a multiline diagram with content $\lambda$,
such that row $2$ 
has configuration $\eta$ and the unique
vacancy in row $1$ is at site $j$,
the conditional probability that row $1$ 
has configuration $\eta'$.
\end{itemize}
\end{lem}

\begin{proof}
Note that by the condition on $\lambda$, both row $2$ and row $1$ 
of the diagram have a single vacant site. 
	It follows from the constraint in \cref{defn:md} that 
the first quantity in \cref{lem:extra-row} is nonzero if and only if 
the second quantity in \cref{lem:extra-row} is nonzero. 

From \eqref{row probability} with $r=1$,
the conditional probability of obtaining a specific 
configuration $\eta'$ on row $1$, given the configuration $\eta$ on row $2$,
is given by 
	$$\prod_{\substack{h:1<h\leq s :\\m_h=1}}  w_D(h,1),$$
where $D$ is any diagram agreeing with $\eta$ and $\eta'$ on 
rows $2$ and $1$ respectively. 

But because of the equivalence of
\eqref{whdef} and \eqref{wDhdef}, that conditional probability is 
exactly the same as \eqref{jump probability}, which is 
the probability of obtaining $\eta'$ when a bell rings at 
site $j$ in the state $\eta$ under the $t$-PushTASEP dynamics.

In the special case where $\eta_j = \eta'_j=0$, i.e. 
both rows $1$ and $2$ of the multiline diagram $D$
have their unique vacancy in position $j$,
 all particles in $D$ must form a trivial match 
between rows $2$ and $1$, and the configuration in the two rows 
is identical. Correspondingly, under the $t$-PushTASEP dynamics,
if the bell rings at the site of an existing vacancy, then the system
stays in its current state. 
\end{proof}

\subsection{ASEP polynomials  from multiline diagrams}
The following result, which is a special case
of a result from \cite{corteel-mandelshtam-williams-2022}, 
relates the distribution 
of the bottom row of a multiline diagram with bottom row $\eta$ to 
the ASEP polynomial indexed by $\eta$.

\begin{thm}
[{\cite[Definition 1.9, Theorem 1.25, Lemma 1.26]{corteel-mandelshtam-williams-2022}}]
\label{thm:ASEP poly combi}
For any composition $\eta =(\eta_1,\dots,\eta_n)$
whose nonzero parts are distinct,
the ASEP polynomial $F_{\eta}(\mathbf{x}; 1, t)$ at $q=1$
can be computed in terms of multiline diagrams as follows:
\begin{equation}\label{multiline-f}
 F_{\eta}(\mathbf{x}; 1, t) = 
\sum_{D:\,\rho^{(1)}(D)=\eta} 
\wt(D) 
\end{equation}
\end{thm}

\begin{rem}
The result from \cite{corteel-mandelshtam-williams-2022}
is more general
than \cref{thm:ASEP poly combi} because it holds for 
any composition and for general $q$. However, we do not need
the more general version in this paper.
\end{rem}

\begin{rem}
Let	$\lambda=(\lambda_1,\dots,\lambda_n)=
	\langle s^{m_s},\dots,0^{m_0} \rangle$.
The ASEP polynomials appearing as numerators of stationary probabilities in \cref{thm:multilrep-rej-statprob} are polynomials in $\mathbf{x}$ with coefficients which are rational functions (but not necessarily polynomials) 
in
 $t$. 
Using the connection with multiline diagrams,
one can show that these probabilities can be rewritten with numerators  that are polynomials in both 
$\mathbf{x}$ and $t$ and with common denominator given by
	\begin{equation}\label{denom}
	P_\lambda(\mathbf{x}; 1, t)
\prod_{i = 1}^s \frac{[m_1 + \cdots + m_i]_t!}{[m_i]_t!}  = 
	e_{\lambda'}(\mathbf{x})
\prod_{i = 1}^s \frac{[m_1 + \cdots + m_i]_t!}{[m_i]_t!}.
	\end{equation} 
In the case of  \cref{eg:210 probs}, this 
common denominator is
\[
(t+1)\left(x_1+x_2+x_3\right) \left(x_1 x_2 + x_1 x_3 + x_2 x_3\right).
\]
The factor in \eqref{denom} involving $t$-factorials is the same as that in \cite{martin-2020} for the ASEP.
\end{rem}

\section{Proof of \cref{thm:multilrep-rej-statprob}}
\label{sec:proof}

In this section we prove 
 \cref{thm:multilrep-rej-statprob}.  We start by proving it in the case where $\lambda$ has distinct parts,
 and then we generalize it to the case of repeated parts, using properties of ASEP polynomials 
 at $q=1$.

\subsection{Proof of \cref{thm:multilrep-rej-statprob}
when $\lambda$ has distinct parts}
\label{sec:distinct types}

\begin{lem}
\label{lem:bottom two lines}
Let $\lambda$ be a partition with distinct entries
and no entry equal to $1$. Consider a 
random multiline diagram $D$ with content $\lambda$, 
with distribution proportional to the weight $\wt(D)$.
The configurations given by the bottom row (row $1$) 
and by the next-to-bottom row (row $2$) 
have the same distribution.
\end{lem}
\begin{proof}
Let $\phi$ be the weakly order-preserving function 
given by $\phi(x)=x-1$ for all $x\geq 1$ and $\phi(0)=0$.
If $s$ is the largest entry of $\lambda$, then
$\phi(\lambda)$ has largest entry $s-1$.

Let $D$ be a random multiline diagram with
content $\lambda$, with distribution proportional
to weight. The diagram $D$ has $s$ rows.
Recall that $\rho^{(1)}$ denotes the sequence of balls in the bottom row of $D$, and $\rho^{(2)}$ denotes the sequence of balls in row $2$
of $D$.

Let $D'$ be a random multiline diagram 
with content $\phi(\lambda)$, again distributed
proportional to weight. The diagram $D'$ has
$s-1$ rows.

In view of the generation process from
 \cref{sec:generate}, 
 if we take rows $s$ down to $2$ of 
the diagram $D$, and subtract $1$ from the label
of every ball, then the resulting diagram
has  distribution identical to that of the diagram $D'$.
In particular, comparing row $2$ of $D$ to row $1$ of $D'$, we have that
$\phi(\rho^{(2)}(D))$ and $\rho^{(1)}(D')$ have the same distribution.

But we may instead compare row $1$ of $D$ to 
row $1$ of $D'$.
By \cref{thm:ASEP poly combi}, 
for 
$\eta\in S_\lambda$, 
the probability that $\rho^{(1)}(D)=\eta$ equals 
	$\frac{\ASEP_\eta(\mathbf{x};1,t)}{\sum_{\tau \in S_{\lambda}} \ASEP_{\tau}(\mathbf{x};1,t)}$,
and 
	the probability that $\rho^{(1)}(D')=\phi(\eta)$ 
equals
	$\frac{\ASEP_{\phi(\eta)}(\mathbf{x};1,t)}{\sum_{\nu \in S_{\phi(\lambda)}} \ASEP_{\nu}(\mathbf{x};1,t)}$.
But $\eta$ and $\phi(\eta)$ have the
same weak standardisation, so by 
\cref{cor:AS}, these two probabilities are equal.
It follows
that $\phi(\rho^{(1)}(D))$ has the same distribution
as $\rho^{(1)}(D')$. 

We have proved that both $\phi(\rho^{(1)}(D))$ and 
$\phi(\rho^{(2)}(D))$ have the same distribution
as $\rho^{(1)}(D')$. But $\phi$ is a bijection from
$S_\lambda$ to $S_{\phi(\lambda)}$. So in 
fact $\rho^{(1)}(D)$ and $\rho^{(2)}(D)$ have the same
distribution, as required.
See \cref{fig:bottom two lines} for an illustration.
\end{proof}

\begin{center}
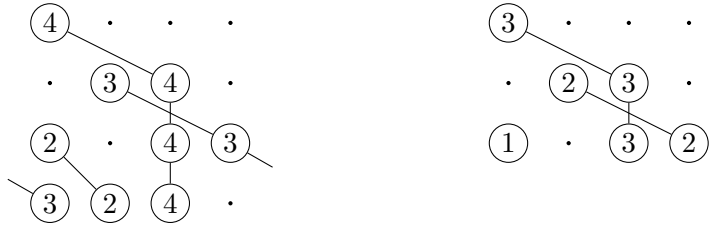
\begin{figure}[h!]

\begin{tikzpicture}[scale=0.8]
\pgfdeclarelayer{background}
\pgfdeclarelayer{foreground}
\pgfsetlayers{background,main,foreground} 

%left-hand picture
% placing vacancy markers in the background
\begin{pgfonlayer}{background}
    \newcommand\rows{4}
    \newcommand\cols{4}

    % Drawing the grid with small dots
    \foreach \row in {1,...,\rows} {
        \foreach \col in {1,...,\cols} {
            \fill (\col,\row) circle (1pt); % Small filled circle
        }
    }
\end{pgfonlayer}

% Placing named balls with numbers on the foreground
\node[draw,circle,fill=white, inner sep=2pt] at (1,4) (ball4-4) {4};
\node[draw,circle,fill=white, inner sep=2pt] at (3,3) (ball4-3) {4};
\node[draw,circle,fill=white, inner sep=2pt] at (3,2) (ball4-2) {4};
\node[draw,circle,fill=white, inner sep=2pt] at (3,1) (ball4-1) {4};

\node[draw,circle,fill=white, inner sep=2pt] at (2,3) (ball3-3) {3};
\node[draw,circle,fill=white, inner sep=2pt] at (4,2) (ball3-2) {3};
\node[draw,circle,fill=white, inner sep=2pt] at (1,1) (ball3-1) {3};

\node[draw,circle,fill=white, inner sep=2pt] at (1,2) (ball2-2) {2};
\node[draw,circle,fill=white, inner sep=2pt] at (2,1) (ball2-1) {2};

% Drawing lines between named balls (wrapping around ring as needed)
\draw (ball4-4) -- (ball4-3);
\draw (ball4-3) -- (ball4-2);
\draw (ball4-2) --(ball4-1);

\draw (ball3-3) -- (ball3-2);
\draw (ball3-2) -- (4.7,1.6);
\draw (0.3,1.4) -- (ball3-1);

\draw (ball2-2) -- (ball2-1);

%right-hand picture
\begin{scope}[xshift=0.5\linewidth]
% placing vacancy markers in the background
\begin{pgfonlayer}{background}
    \newcommand\rows{3}
    \newcommand\cols{4}

    % Drawing the grid with small dots
    \foreach \row in {1,...,\rows} {
        \foreach \col in {1,...,\cols} {
            \fill (\col,\row+1) circle (1pt); % Small filled circle
        }
    }
\end{pgfonlayer}
% Placing named balls with numbers on the foreground
\node[draw,circle,fill=white, inner sep=2pt] at (1,4) (ball3-3) {3};
\node[draw,circle,fill=white, inner sep=2pt] at (3,3) (ball3-2) {3};
\node[draw,circle,fill=white, inner sep=2pt] at (3,2) (ball3-1) {3};

\node[draw,circle,fill=white, inner sep=2pt] at (2,3) (ball2-2) {2};
\node[draw,circle,fill=white, inner sep=2pt] at (4,2) (ball2-1) {2};

\node[draw,circle,fill=white, inner sep=2pt] at (1,2) (ball1-1) {1};

% Drawing lines between named balls (wrapping around ring as needed)
\draw (ball3-3) -- (ball3-2);
\draw (ball3-2) -- (ball3-1);

\draw (ball2-2) -- (ball2-1);

\end{scope}

\end{tikzpicture}

\caption{
On the left, a multiline diagram $D$ with content
$\lambda=(4,3,2,0)$. On the right, a multiline diagram $D$
with content $\phi(\lambda)=(3,2,1,0)$ 
(where $\phi$ is defined as in the proof of 
\cref{lem:bottom two lines}). The configurations
$\rho^{(2)}(D)$ and $\rho^{(1)}(D)$ (the two lowest
rows of $D$) have the same distribution, 
and the distribution of $\phi(\rho^{(1)}(D))$ 
and of $\phi(\rho^{(2)}(D))$ is the same as that of $\rho^{(1)}(D')$.
\label{fig:bottom two lines}}
\end{figure}
\end{center}

Write $p_j(\eta, \eta')$ for the 
probability of obtaining $\eta'$ from $\eta$ 
using the $t$-PushTASEP jump dynamics when a bell rings at site $j$.

Now we average over $j$. Specifically, we take 
a weighted average of $p_j$ with weights proportional 
to $x_j^{-1}$:
\begin{equation}
\label{pdef}
p(\eta, \eta')
=
\sum_j
\frac{x_j^{-1}}{x_1^{-1}+x_2^{-1}+\dots+x_n^{-1}}
p_j(\eta, \eta').
\end{equation}

Since $x_j^{-1}$ is the rate at which the bell at site $j$ in the $t$-PushTASEP, we have that 
$p$ defined by \eqref{pdef} gives the transition probabilities of a 
discrete-time Markov chain whose stationary distribution 
is the same as that of the continuous-time $t$-PushTASEP. 

By \cref{lem:extra-row}, $p_j(\eta, \eta')$ also gives
the probability of obtaining $\eta'$ as row $1$
of a multiline diagram with content $\lambda$, given that row 
$2$ is $\eta$ and that the vacancy in the bottom row is at site $j$. But the probability of having a vacancy at $j$ is proportional to $x_j^{-1}$, independently of the rows above, so $p(\eta, \eta')$ defined by \eqref{pdef} is the probability of obtaining $\eta'$ as row $1$ of the diagram, given that row $2$ is $\eta$.

We also know from \cref{lem:bottom two lines}
that the distributions of row $1$ and row $2$ are the same. 
Since row $1$ is obtained from row $2$ by a single update of
the dynamics $p$, this tells us that their common distribution is a stationary distribution for $p$, and hence also for the 
continuous-time $t$-PushTASEP. 

But this common distribution is also proportional 
to the ASEP polynomials. So we obtain that 
$\pi_{\lambda}(\eta), \eta\in S_\lambda$ is proportional to $F_\eta$ as required to give the result
of \cref{thm:multilrep-rej-statprob} in this case.

\subsection{Proof of \cref{thm:multilrep-rej-statprob} in the general case}

To extend the result from the previous section
to prove the general case of \cref{thm:multilrep-rej-statprob}, we apply the results from \cref{sec:symm-fn}.

Consider some partition $\mu$ (as ever, assumed to have
at least one entry $0$). We can find some $\lambda$ 
satisfying the conditions 
of \cref{sec:distinct types} 
(i.e.\ having distinct entries, no entry $1$, and
exactly one entry $0$) and some 
weakly order-preserving function $\phi$, 
such that $\mu=\phi(\lambda)$. 
\cref{prop:colouring-stationary}
then tells us
that 
\begin{equation}\label{eq:sum}
	\pi_{\mu}(\eta)
=
\sum_{\zeta\in S_\lambda: \phi(\zeta)=\eta}
	\pi_{\lambda}(\zeta).
\end{equation}

{But we also know from 
\cref{thm:multilrep-rej-statprob} (which we now know holds 
in the case that $\lambda$ has distinct parts)
that for $\zeta\in S_\lambda$, 
$$\pi_{\lambda}(\zeta) = 
\frac{\ASEP_\zeta(\mathbf{x};1, t)}{P_\lambda(\mathbf{x}; 1, t)}.$$
We then obtain from \eqref{eq:sum} that 
$$\pi_{\mu}(\eta)
=
\sum_{\substack{\zeta\in S_\lambda \\ \phi(\zeta)=\eta}}
\frac{\ASEP_\zeta(\mathbf{x};1, t)}{P_\lambda(\mathbf{x}; 1, t)}.$$
Now by \eqref{eq:G} we have that 
$$\pi_{\mu}(\eta)
=
\frac{G_\eta(\mathbf{x}, t)}{P_\lambda(\mathbf{x}; 1, t)}.$$
Finally using \cref{prop:f-projection}, we 
 obtain that 
$$\pi_{\mu}(\eta)
=
\frac{F_\eta(\mathbf{x}, t)}{P_\mu(\mathbf{x}; 1, t)},$$ as desired.
This completes the proof of \cref{thm:multilrep-rej-statprob}.}

\begin{proof}[Proof of \cref{thm:obs}]
Let $A\subset S_\lambda$ be any event (i.e.\ collection of configurations)
that is conserved by exchanging the contents of sites $i$ and $i+1$; that is, $A=s_i A = \{s_i\eta \mid \eta\in A\}$.

Since $s_i$ is a bijection on $S_\lambda$, we then have
$\sum_{\eta\in A} F_\eta = \sum_{\eta\in A} F_{s_i\eta}$,
giving
\[
\sum_{\eta\in A} F_\eta = \frac12 \sum_{\eta\in A} 
\left(F_\eta+F_{s_i\eta}\right).
\]
By \cref{lem:ASEP pol sym}, this quantity is symmetric 
in $x_i$ and $x_{i+1}$.

Now suppose $A$ depends only on the configuration in sites 
$1,2,\dots, k$. Then the above holds for any $i$ with $k<i<n$. 
We obtain that for any such $i$, the probability of $A$
in the stationary distribution,
\[
p^{(\lambda)}(A)
:= 
\sum_{\eta\in A} p^{(\lambda)}_\eta
=
\sum_{\eta\in A} 
\frac{F_\eta(x_1,\dots, x_n; 1, t)}
{P_\lambda(\mathbf{x}; 1, t)},
\]
is symmetric in $x_i$ and $x_{i+1}$
(since the Macdonald polynomial in the denominator is 
symmetric). But then the probability is in fact symmetric in 
all of $x_{k+1}, \dots, x_n$. 
This gives the symmetry required for \cref{thm:obs}.
\end{proof}

\section{Formulas for density and currents}
\label{sec:current}

In this section we discuss the density of particles and the current in the 
$t$-PushTASEP.
In particular, we give a formula for the 
density of particles in \cref{cor:dens}.  We also prove 
\cref{thm:curr-singlePushTASEP}, which gives a formula for the current
in the single species case.  We end with a discussion
of the current in the multispecies case.

Let us consider the single species $t$-PushTASEP first.  
By \cref{prop:ss-singlePushTASEP},
the 
density {(in the stationary distribution)} of particles is the same as for the PushTASEP at $t=0$.
As we will explain, 
this continues to hold for the multispecies $t$-PushTASEP.
Let $\eta^{(i)}_j$ denote the occupation variable for the particle of species $i$ at site $j$, i.e $\eta^{(i)}_j = 1$ (resp. $\eta^{(i)}_j=0$) provided the $j$'th site is occupied (resp. not occupied) by $i$. 
The formula for the density of particles in the multispecies PushTASEP is obtained directly from that in the single species case:
the density of the particle of species $j$ is the density of the particle of species $1$ in the single species PushTASEP with $m_j+\cdots+m_s$ particles minus the density of the particle of species $1$ in the single species PushTASEP with $m_{j+1}+\cdots+m_{s}$ particles. 
The formula for the density,
shown in \cref{cor:dens} below,
is the same as given for the $t=0$ case
in \cite[Proposition 18]{ayyer-martin-2023}.
The proof is identical and is omitted.

\begin{cor}
\label{cor:dens}
The density of species $r$ at the first site in the multispecies $t$-PushTASEP with content $\lambda = \langle s^{m_s},\dots,0^{m_0} \rangle$ on $n = \sum_i m_i$ sites is given by
\[
\aver{\eta^{(r)}_1} = x_1 \frac{s_{\langle 2^{a_{r+1}}, 1^{m_r-1} \rangle}(x_2,\dots,x_n)}
{e_{a_r}(\mathbf{x}) e_{a_{r+1}}(\mathbf{x})},
\]
where $s_\mu$ is the Schur polynomial indexed by $\mu$, and $a_r = m_r + \cdots + m_s$ for $1 \leq r \leq s$.
\end{cor}

We now move on to studying the current (at stationarity) for the single species $t$-PushTASEP.
Let $\lambda = \langle 1^{m_1}, 0^{m_0} \rangle$, where
$m_1+m_0=n$, so that we are studying a system
with $m_1$ particles and $m_0$ vacancies. 
Recall that the \emph{current (at stationarity)} of a 
particle across a given edge (say $(n,1)$) 
of the lattice is the number of particles per 
unit time that cross that edge in the long-time limit. 
Because of particle conservation, the current 
is independent of the edge.  We will denote the stationary
current in our system of particles by $J_{m_0, m_1}$.

In terms of the stationary distribution for the $t$-PushTASEP, 
we can compute $J_{m_0,m_1}$  as follows. If a particle is at position $k$ and there is a vacancy at position $j < k$, then the particle at $k$ can make a transition to $j$, and contributes to the current across the edge $(n,1)$ in doing so. If there are $h$ vacancies in sites $k+1, \dots, n, 1, \dots, j-1$, then the rate of this transition is $t^h/x_k$. Formally, we can write the current as
\begin{equation}
\label{def-curr}
J_{m_0, m_1} 
= \sum_{h = 0}^{m_0 - 1} \sum_{j = 1}^{n-m_0+h} \sum_{k = j+1}^n
\frac{t^h}{x_k} 
	\sum_{\eta}
	\pi(\eta),
\end{equation}
where the sum on the right is over all 
$$\{\eta \in S_{\langle 1^{m_1},0^{m_0} \rangle} \ \vert \ 
\eta_k=1, \eta_j=0, \text{ and }\eta \text{ has $h$ vacancies in 
sites $k+1,\dots, n,1 \dots  j-1$}\}.$$

The formula we need to prove \cref{thm:curr-singlePushTASEP} is the following identity for elementary symmetric functions, which seems to be new.

\begin{lem}
\label{lem:e-identity}
Fix $n$ and $m_0 < n$ positive integers. Then, for all $0 \leq h \leq m_0 - 1$, we have
\begin{multline}
\label{e-identity}
 (h+1) e_{n-m_0-1}(x_1,\dots,x_n) = \\
\sum_{a = m_0 - h}^{n-1} \sum_{j = 1}^{n-a}
e_{h-m_0+a}(x_{j+1}, \dots, x_{j+a-1}) 
\times e_{n-h-1-a}(x_{j+a+1}, \dots, x_n, x_1, \dots, x_{j-1}). 
\end{multline}
\end{lem}

\begin{proof}
For convenience, set $m_1 = n - m_0$. 
We need to show that every monomial in $e_{m_1-1}$ occurs on the right hand side exactly $h+1$ times.
So fix a subset $S = \{1 \leq s_1 < \cdots < s_{m_1-1} \leq n\}$ and consider the summand in the right hand side of \eqref{e-identity}. It depends on two parameters $j$ and $a$ and we set $k = j+a$. Thus, the sum depends on two parameters, $j$ and $k$ instead.
The first factor in the summand is an elementary symmetric function depending on the variables strictly between $j$ and $k$, and the second is one depending on variables between $k$ and $j$ counted cyclically. Therefore, we must choose $j$ and $k$ to be in the set $[n] \setminus S$ such that there are $h+k-j-m_0$ elements of $S$ between $j$ and $k$. To complete the proof, it will suffice to show that there are exactly $h+1$ many choices.

To make the argument easier to follow, let us first consider the case where $h = m_0 - 1$. Then, we have to choose $j$ and $k$ so that there is no element of $[n] \setminus S$ strictly between them. There is exactly one way of choosing $j \in [1, s_1 - 2]$ and that is with $k = j+1$. We can also choose $j = s_1 - 1$ with $k = s_1 + 1$. Thus, $j$ can be chosen to be any position between $1$ and $s_1 - 1$ in exactly one way. Similarly, $j$ can be chosen to be any position between $s_i+1$ and $s_{i+1} - 1$ in one way, for $1 \leq i \leq m_1 - 2$. Lastly, $j$ can be chosen to be any element between 
$s_{m_1-1}+1$ and $n - 1$ with $k = j+1$. Summing all of these possibilities, we get $n - m_1 = m_0 = h+1$, which is independent of $S$.

The argument for general $h$ goes the same way. Between each $s_i+1$ and $s_{i+1} - 1$, there is exactly one way of choosing $k$ so that there are $h+k-j-m_0$ elements of $S$ between $j$ and $k$, for small values of $i$. The change occurs at the end as we get closer to $n$. Every time the value of $h$ increases by $1$, the number of possibilities of $j$ reduces exactly by $1$. It is easy to see that this argument is independent of the choice of $S$.

As a sanity check, consider the case of $h = 0$. In that case, the only possibility is to choose $j$ (resp. $k$) to be the smallest (resp. largest) element of $[n] \setminus S$, which is consistent with what we want to prove.
\end{proof}

\begin{proof}[Proof of \cref{thm:curr-singlePushTASEP}]
The current is independent of the edge being considered. So look at the edge connecting $n$ to $1$. For a particle hop to count towards the current across this edge, it must hop from a site $k \in [2, n]$ to a vacant site $j < k$.
If there are $h$ holes between the sites $k+1$ to $j-1$ (of which there are $n-k+j-1$ many), then the rate of this transition is $t^{h}/([m_0] x_k)$. Therefore the total stationary weight of these configurations is 
$e_{n-h-1-k+j}(x_{k+1}, \dots, x_n, x_1, \dots, x_{j-1})$.
Similarly, there are $m_0-h-1$ holes between the $k-j-1$ sites between $j+1$ and $k-1$ and so the total stationary weight of such configurations is
$e_{h-m_0+k-j}(x_{j+1}, \dots, x_{j+a-1})$. Summing over all possible values of $j$ and $k$, we see that the current is
\begin{multline*}
J_{m_0, m_1} = \sum_{h = 0}^{m_0 - 1} \sum_{j = 1}^{n-m_0+h} 
\sum_{k = j+1}^n \frac{t^h}{[m_0]}
\frac{e_{h-m_0+k-j}(x_{j+1}, \dots, x_{j+a-1})}
{e_{m_1}(x_1,\dots,x_{n})} \\
\times e_{n-h-1-k+j}(x_{k+1}, \dots, x_n, x_1, \dots, x_{j-1}).
\end{multline*} 
Now substituting $a =j-k$ and using \cref{lem:e-identity}, we arrive at
\[
J_{m_0, m_1}  = \frac{e_{m_1-1}(x_1,\dots, x_{n}) }{e_{m_1}(x_1,\dots,x_{n})} \sum_{h = 0}^{m_0 - 1}  \frac{(h+1)t^h}{[m_0]},
\]
which gives the desired result.
\end{proof}

Now we would like to compute the current for the multispecies case. The current of species $j$ in the $t$-PushTASEP on $\Omega_\lambda$ is the difference of the total currents of species $j$ through $s$ minus the total currents of species $j+1$ through $s$. Following the argument in \cite[Proposition 19]{ayyer-martin-2023}, we would like to calculate both these in terms of the single species $t$-PushTASEP using \cref{thm:curr-singlePushTASEP} and \cref{prop:colouring}. 
Unfortunately, this does not work as in the $t = 0$ case if $j < s$. The main reason is that an edge can contribute towards multiple currents in a single transition when $t > 0$. 

We illustrate this with the example of $\lambda = (2, 1, 0)$ shown in \cref{fig:eg123t}. 
Consider the current of species $1$ across the edge $(3, 1)$. Using the colouring argument, we would get this current to be
\[
J_{1,2} - J_{2, 1}
= \frac{e_1(x_1, x_2, x_3)}{e_1(x_1, x_2, x_3)}
= \frac{1 + 2t}{1 + t} \frac{e_0(x_1, x_2, x_3)}
{e_1(x_1, x_2, x_3)},
\]
which gives, after some manipulations,
\[
J_{1,2} - J_{2, 1}
= 
\frac{(1+t)(x_1^2 + x_2^2 + x_3^2)
+ (x_1 x_2 + x_1 x_3 + x_2 x_3)}
{(1+t) e_{2,1}(x_1, x_2, x_3)}.
\]
Now let us calculate the current by brute force. Particle $1$ jumps across the edge $(3, 1)$ only for the following states when either the $1$ jumps, or when the $2$ jumps displacing the $1$:
\begin{itemize}
\item $(0,1,2)$,
\item $(0,2,1)$, 
\item $(2,0,1)$.
\end{itemize}
The sum of these contributions gives
\begin{multline*}
\frac{x_2 x_3}{e_{2,1}(x_1, x_2, x_3)} 
\left(\frac{x_1 t}{1+t} + x_3 \right) 
\left(\frac{1}{x_2} + \frac{t}{x_3(1+t)} \right) \\
+ \frac{x_2 x_3}{e_{2,1}(x_1, x_2, x_3)} 
\left(\frac{x_1}{1+t} + x_2\right)
\left(\frac{1}{x_3} + \frac{1}{x_2(1+t)} \right) \\
+ \frac{x_1 x_3}{e_{2,1}(x_1, x_2, x_3)} 
\left(x_1 + \frac{x_2 t}{1+t} \right)
\left(\frac{1}{x_3} + \frac{t}{x_1(1+t)} \right),
\end{multline*}
which after simplifying becomes
\[
\frac{(1+t)^2(x_1^2 + x_2^2 + x_3^2)
+ (1+2t+2t^2)(x_1 x_2 + x_1 x_3 + x_2, x_3)}
{(1+t)^2 e_{2,1}(x_1, x_2, x_3)},
\]
and this does not match $J_{1,2} - J_{2, 1}$. The main reason is that (i) in the transition from $(0,1,2)$ where the $2$ displaces the $1$, both particles end up crossing the edge $(3,1)$.

\bibliography{lrep}

\begin{thebibliography}{CdGW20}

\bibitem[Ale19]{alexandersson-2019}
Per Alexandersson.
\newblock Non-symmetric {M}acdonald polynomials and {D}emazure-{L}usztig
  operators.
\newblock {\em S\'{e}m. Lothar. Combin.}, 76:Art. B76d, 27, [2016--2019].

\bibitem[AM23]{ayyer-martin-2023}
Arvind Ayyer and James Martin.
\newblock The inhomogeneous multispecies {PushTASEP}: Dynamics and symmetry.
\newblock Preprint at \url{https://arxiv.org/abs/2310.09740}, 2023.

\bibitem[AMM22]{ayyer-mandelshtam-martin-2022}
Arvind Ayyer, Olya Mandelshtam, and James~B. Martin.
\newblock Modified {M}acdonald polynomials and the multispecies zero range
  process: {II}, 2022.
\newblock Preprint at \url{https://arxiv.org/abs/2209.09859}.

\bibitem[ANP23]{aggarwal-nicoletti-petrov-2023}
Amol Aggarwal, Matthew Nicoletti, and Leonid Petrov.
\newblock Colored interacting particle systems on the ring: Stationary measures
  from {Y}ang-{B}axter equation, 2023.
\newblock Preprint at \url{https://arxiv.org/abs/2309.11865}.

\bibitem[AS19]{alexandersson-sawhney-2019}
Per Alexandersson and Mehtaab Sawhney.
\newblock Properties of non-symmetric {M}acdonald polynomials at {$q=1$} and
  {$q=0$}.
\newblock {\em Ann. Comb.}, 23(2):219--239, 2019.

\bibitem[BW22]{borodin-wheeler-2022}
Alexei Borodin and Michael Wheeler.
\newblock {\em Coloured stochastic vertex models and their spectral theory},
  volume 437 of {\em Ast{\'e}risque}.
\newblock Paris: Soci{\'e}t{\'e} Math{\'e}matique de France (SMF), 2022.

\bibitem[CdGW15]{cantini-degier-wheeler-2015}
Luigi Cantini, Jan de~Gier, and Michael Wheeler.
\newblock Matrix product formula for {M}acdonald polynomials.
\newblock {\em J. Phys. A}, 48(38):384001, 2015.

\bibitem[CdGW20]{chen-degier-wheeler-2020}
Zeying Chen, Jan de~Gier, and Michael Wheeler.
\newblock Integrable stochastic dualities and the deformed
  {K}nizhnik-{Z}amolodchikov equation.
\newblock {\em Internat. Math. Res. Notices}, 2020(19):5872--5925, 2020.

\bibitem[Che95]{cherednik-1995b}
Ivan Cherednik.
\newblock Nonsymmetric {M}acdonald polynomials.
\newblock {\em Internat. Math. Res. Notices}, (10):483--515, 1995.

\bibitem[CMW22]{corteel-mandelshtam-williams-2022}
Sylvie Corteel, Olya Mandelshtam, and Lauren Williams.
\newblock From multiline queues to {M}acdonald polynomials via the exclusion
  process.
\newblock {\em Amer. J. Math.}, 144(2):395--436, 2022.

\bibitem[Fer11]{ferreira-2011}
Jeffrey~Paul Ferreira.
\newblock {\em Row-strict {Q}uasisymmetric {S}chur {F}unctions,
  {C}haracterizations of {D}emazure {A}toms, and {P}ermuted {B}asement
  {N}onsymmetric {M}acdonald {P}olynomials}.
\newblock ProQuest LLC, Ann Arbor, MI, 2011.
\newblock PhD thesis, University of California, Davis.

\bibitem[FM06]{ferrari-martin-2006}
Pablo~A. Ferrari and James~B. Martin.
\newblock Multi-class processes, dual points and {$M/M/1$} queues.
\newblock {\em Markov Process. Related Fields}, 12(2):175--201, 2006.

\bibitem[FM07]{ferrari-martin-2007}
Pablo~A. Ferrari and James~B. Martin.
\newblock Stationary distributions of multi-type totally asymmetric exclusion
  processes.
\newblock {\em Ann. Probab.}, 35(3):807--832, 2007.

\bibitem[FM09]{ferrari-martin-2009}
Pablo~A. Ferrari and James~B. Martin.
\newblock Multiclass {H}ammersley-{A}ldous-{D}iaconis process and
  multiclass-customer queues.
\newblock {\em Ann. Inst. Henri Poincar\'{e} Probab. Stat.}, 45(1):250--265,
  2009.

\bibitem[HHL08]{HHL3}
J.~Haglund, M.~Haiman, and N.~Loehr.
\newblock A combinatorial formula for nonsymmetric {M}acdonald polynomials.
\newblock {\em Amer. J. Math.}, 130(2):359--383, 2008.

\bibitem[KT07]{KasataniTakeyama}
Masahiro Kasatani and Yoshihiro Takeyama.
\newblock The quantum {K}nizhnik-{Z}amolodchikov equation and non-symmetric
  {M}acdonald polynomials.
\newblock {\em Funkcial. Ekvac.}, 50(3):491--509, 2007.

\bibitem[Lig85]{liggett-1985}
Thomas~M. Liggett.
\newblock {\em Interacting particle systems}, volume 276 of {\em Grundlehren
  der mathematischen Wissenschaften [Fundamental Principles of Mathematical
  Sciences]}.
\newblock Springer-Verlag, New York, 1985.

\bibitem[Mac95]{Macdonald}
Ian Macdonald.
\newblock {\em Symmetric functions and {H}all polynomials}.
\newblock Oxford Mathematical Monographs. The Clarendon Press, Oxford
  University Press, New York, second edition, 1995.
\newblock With contributions by A. Zelevinsky, Oxford Science Publications.

\bibitem[Mar99]{marshall-1999}
Dan Marshall.
\newblock Symmetric and nonsymmetric {M}acdonald polynomials.
\newblock {\em Ann. Comb.}, 3(2-4):385--415, 1999.

\bibitem[Mar20]{martin-2020}
James~B. Martin.
\newblock Stationary distributions of the multi-type {A}{S}{E}{P}.
\newblock {\em Electron. J. Probab.}, 25:1--41, 2020.

\bibitem[PEM09]{prolhac-evans-mallick-2009}
S~Prolhac, M~R Evans, and K~Mallick.
\newblock The matrix product solution of the multispecies partially asymmetric
  exclusion process.
\newblock {\em J. Phys. A}, 42(16):165004, 2009.

\bibitem[Pet20]{petrov-2020}
Leonid Petrov.
\newblock {PushTASEP in inhomogeneous space}.
\newblock {\em Electron. J. Probab.}, 25:1 -- 25, 2020.

\bibitem[Spi70]{spitzer-1970}
Frank Spitzer.
\newblock Interaction of {M}arkov processes.
\newblock {\em Adv. Math.}, 5(2):246 -- 290, 1970.

\end{thebibliography}
\bibliographystyle{alpha}

\end{document}